\documentclass[10pt]{article}

\usepackage[T1]{fontenc}
\usepackage{lmodern}

\usepackage{amsmath, amssymb, amsthm}
\usepackage[margin=2.5cm]{geometry}
\usepackage{color, graphicx, float}
\usepackage{subcaption}
\usepackage{changepage}
\usepackage[breakable]{tcolorbox}
\usepackage{multirow}
\usepackage{tabularray}
\usepackage{lscape} 
\usepackage{tabularx}
\usepackage{rotating}
\usepackage{etoolbox}
\makeatletter
\patchcmd{\@maketitle}{\LARGE \@title}{\LARGE\bfseries\@title}{}{}

\renewcommand{\@seccntformat}[1]{\csname the#1\endcsname.\quad}
\makeatother

\definecolor{darkblue}{rgb}{0,0,.5}

\usepackage{hyperref}
\hypersetup{
	colorlinks=true,		
	linkcolor=darkblue,		
	citecolor=darkblue,		
	urlcolor=darkblue 		
}

\makeatletter
\def\th@plain{%
	\thm@notefont{}
	\itshape 
}
\def\th@definition{%
	\thm@notefont{}
	\normalfont 
}

\renewenvironment{proof}[1][\proofname]{\par
	\normalfont
	\topsep0\p@\@plus3\p@ \trivlist
	\item[\hskip\labelsep\itshape
	#1\@addpunct{.}]\ignorespaces
}{%
	\qed\endtrivlist
}
\makeatother

\newtheorem{theorem}{Theorem}[section]
\newtheorem{lemma}[theorem]{Lemma}

\newtheorem{proposition}[theorem]{Proposition}
\theoremstyle{definition}

\theoremstyle{definition}

\theoremstyle{definition}
\newtheorem{remark}[theorem]{Remark}
\theoremstyle{definition}
\newtheorem{algorithm}[theorem]{Algorithm}
\theoremstyle{definition}
\newtheorem{assumption}[theorem]{Assumption}

\usepackage[shortlabels]{enumitem}

\newcommand{\dom}{\ensuremath{\operatorname{dom}}}
\newcommand{\prox}{\ensuremath{\operatorname{Prox}}}
\newcommand{\argmin}{\ensuremath{\operatorname*{argmin}}}
\newcommand{\epi}{\ensuremath{\operatorname{epi}}}
\newcommand{\dist}{\ensuremath{\operatorname{dist}}}

\usepackage{threeparttable}

\allowdisplaybreaks

\newcounter{step}
\newcommand\step[1]{%
	\refstepcounter{step}	
	\vskip 0.25\baselineskip
	\ifx\hfuzz#1\hfuzz
		\item[~\(\triangleright\)~\textbf{Step~\arabic{step}.}]
	\else
		\item[~\(\triangleright\)~\textbf{Step~\arabic{step}}] (\texttt{#1})\textbf{.}%
	\fi
}

\begin{document}

\title{A proximal splitting algorithm for generalized DC programming with applications in signal recovery}

\author{
Tan Nhat Pham\thanks{Centre for New Energy Transition Research, Federation University Australia, Ballarat, VIC 3353, Australia.
E-mail: \texttt{pntan.iac@gmail.com}. ORCID: \href{https://orcid.org/0000-0002-0445-745X}{0000-0002-0445-745X}.}, 
~
Minh N. Dao\thanks{School of Science, RMIT University, Melbourne, VIC 3000, Australia.
E-mail: \texttt{minh.dao@rmit.edu.au}. ORCID: \href{https://orcid.org/0000-0002-8074-6675}{0000-0002-8074-6675}.
},
~
Nima Amjady\thanks{Centre for New Energy Transition Research, Federation University Australia, Ballarat, VIC 3353, Australia.
E-mail: \texttt{n.amjady@federation.edu.au}. ORCID: \href{https://orcid.org/0000-0003-1308-1738}{0000-0003-1308-1738}.},
~
Rakibuzzaman Shah\thanks{Centre for New Energy Transition Research, Federation University Australia, Ballarat, VIC 3353, Australia.
E-mail: \texttt{m.shah@federation.edu.au}. ORCID: \href{https://orcid.org/0000-0001-8314-0225}{0000-0001-8314-0225}.}
}

\date{April 28, 2025}

\maketitle

\begin{abstract}
The difference-of-convex (DC) program is {\color{black}an important model in} nonconvex optimization due to its structure, which encompasses a wide range of practical applications. In this paper, we aim to tackle a generalized class of DC programs, where the objective function is formed by summing a possibly nonsmooth nonconvex function and a differentiable nonconvex function with Lipschitz continuous gradient, and then subtracting a nonsmooth continuous convex function. We develop a proximal splitting algorithm that utilizes proximal evaluation for the concave part and Douglas--Rachford splitting for the remaining components. The algorithm guarantees subsequential convergence to a {\color{black}critical} point of the problem model. Under the widely used Kurdyka--{\L}ojasiewicz property, we establish global convergence of the full sequence of iterates and derive convergence rates for both the iterates and the objective function values, without assuming the concave part is differentiable. The performance of the proposed algorithm is tested on signal recovery problems with a nonconvex regularization term and exhibits competitive results compared to notable algorithms in the literature on both synthetic data and real-world data\footnote{MATLAB codes and data used in this work is available at: \url{https://github.com/nhattan214/Backward_Douglas_Rachford}}.
\end{abstract}


\smallskip
\noindent{\bf Mathematics Subject Classification (MSC 2020):}
90C26,	
49M27,	
65K05.	

\section{Introduction}
Throughout this paper, the set of nonnegative integers is denoted by $\mathbb{N}$, the set of real numbers  is denoted by $\mathbb{R}$, the set of nonnegative real numbers is denoted by $\mathbb{R}_+$, and the set of the positive real numbers is denoted by $\mathbb{R}_{++}$. We use $\mathbb{R}^d$ to represent the $d$-dimensional Euclidean space with inner product $\langle \cdot, \cdot \rangle$ and Euclidean norm $\|\cdot\|$.

We consider the structured optimization problem
\begin{equation}\label{eq:P}
\min_{x\in \mathbb{R}^d} {\color{black}\mathbf{F}(x)} :=f(x)+h(x)-g(x), 
\tag{P}
\end{equation}
where $f\colon \mathbb{R}^d\to \mathbb{R}$ is a differentiable (possibly nonconvex) function whose gradient is Lipschitz continuous, $h\colon \mathbb{R}^d\to (-\infty, +\infty]$ is a lower semicontinuous function, and $g\colon \mathbb{R}^d \to \mathbb{R}$ is a continuous convex function (possibly nonsmooth). 

When all the functions are convex, problem \eqref{eq:P} becomes {\color{black} the structured problem called the \emph{generalized difference-of-convex} (DC) \emph{program} in \cite{Chuang2021}}, which reduces to the standard DC program introduced in the original work \cite{Tao1986} when {\color{black}$f\equiv 0$}. The generalized DC program encompasses a wide range of practical engineering problems, ranging from image processing, portfolio optimization, signal processing, to machine learning, as shown in \cite{Le_Thi2018-ht} and references therein. {\color{black} DC programs also have applications in operation research, such as finding continuous piecewise linear approximations of deterministic functions \cite{Kazda2024}, or risk-based robust statistical learning \cite{Liu2023OR}}. A classical algorithm to solve this class of problems is the \emph{difference-of-convex algorithm} (DCA) and its variants, as thoroughly reviewed in \cite{Le_Thi2018-ht}. At each iteration, the DCA replaces the concave part of the objective by a linear majorant and solves the resulting convex optimization problem. The choice of DC decomposition of the objective function heavily affects the complexity of the subproblems. {\color{black} There are also research works that aim at accelerating convergence of the DCA, notably the \textit{boosted DC algorithm} \cite{AragnArtacho2020}}. Proximal algorithm was also developed to solve the DC program, as introduced in \cite{PDCA_first}. This algorithm majorizes the concave part in the objective by a linear majorant in each iteration. A notable extended version of proximal algorithm for the DC program is the \textit{proximal difference-of-convex algorithm with extrapolation} (pDCAe) \cite{Wen2017}, in which the authors considered a class of problem having a similar structure to problem \eqref{eq:P} with all convex components. There are also other works on the DC program which relax the convexity of the terms in the objective function, such as those in \cite{An2016, Tan2023,vanAckooij2019,deOliveira2025}.

The DC program also attracts researchers to develop other types of splitting algorithms. The \emph{alternating direction method of multipliers}, or ADMM, is also a powerful algorithm since it splits the original problem into several subproblems with smaller scale. The ADMM is originally designed for sum of two functions, and generally the subproblems may not have closed-form solutions. Therefore, external solvers or numerical methods need to be used to solve the subproblem, resulting in higher computational complexity. To overcome this, Sun et al. \cite{Sun2017} developed a variant of the ADMM to solve the DC program, which incorporates Bregman distance to the subproblems, while replacing the concave part with a linear majorant (BADMM-DC). By doing so, the subproblems of the BADMM-DC are convex, hence they can be easily solved. Tu et al. \cite{Tu2019} proposed an improved version of the BADMM-DC by incorporating extrapolation techniques, and achieved better performance. Another splitting algorithm which attracts significant attention in the literature is the \emph{Douglas--Rachford} (DR) \emph{splitting} \cite{DR56}, originally designed for finding zeros in the sum of two maximally monotone operators. The DR scheme was further extended to what is known as the \emph{Davis--Yin splitting} \cite{Davis2017} to deal with the sum of three functions, one of which has Lipschitz continuous gradient. Based on this extension, the so-called \emph{unified DR algorithm} was proposed in \cite{Chuang2021} to solve the DC program, where the authors employed monotone operator theory to prove the convergence and required strong convexity of the summands as well as smoothness of the concave part. This leads to one drawback that although the unified DR algorithm can cover the classical DR algorithm, its convergence proof cannot cover that of the classical DR algorithm, {\color{black} which do not require the strong convexity of $f$ and $h$} \cite[{\color{black}Remark 3.1}]{Chuang2021}. Despite the fact that this work has shown a promising splitting scheme to solve the DC program, the assumptions required for the convergence are still strong.

A common approach to establish the convergence in nonconvex setting is to use the \emph{Kurdyka--Łojasiewicz} (KL) \emph{property} \cite{Kurdyka1998,Loja63}, as shown in the framework \cite{Attouch2011}. This famous framework particularly states that if a suitable \emph{Lyapunov function} (also known as a merit or potential function), with critical points related to those of the original problem, satisfies the KL property, then the entire sequence generated by the algorithm can be shown to converge to a critical point, rather than just a subsequence. This has been extensively used in many previous works, including those dealing with the DC program such as \cite{An2016,Wen2017,Tan2023,Sun2017}. One worth-noting point from these works is that, in order to establish the convergence of the entire sequence under the KL property, they need an additional assumption that the concave part has Lipschitz continuous gradient, which may not be satisfied in practice. To overcome this assumption, the authors in \cite{Banert2018} evaluate the concave part via proximal steps instead of using its subgradients, thereby establishing the full sequential convergence without the additional assumption. It is also discussed in \cite{Banert2018} that using proximal steps in an algorithm offers several advantages over using subgradients, such as better convergence rates and avoiding the issue of the algorithm getting stuck when the subdifferential at a certain point is a non-singleton set.

In this work, we develop a proximal splitting algorithm that incorporates the Fenchel conjugate to handle the concave part in solving problem \eqref{eq:P}. The convexity requirement in the problem \eqref{eq:P} is weaker than those used in the literature for the general DC program, making it a more general class of optimization problems. 
{\color{black} Compared to the closely related work in \cite{Chuang2021}, our convergence analysis does not rely on monotone operator theory, and hence we can bypass the strong convexity of $f$ and $h$ as well as the smoothness of $g$. In particular, our work only requires $g$ to be convex but possibly nonsmooth, while there is no convexity requirement for $f$ and $h$}. We construct a new Lyapunov function to prove that the generated sequence is bounded and each of its cluster points yields a {\color{black}critical point} of problem \eqref{eq:P}. When the Lyapunov function satisfies the KL property, the convergence of the full sequence is established, and the convergence rates of this sequence as well as those of the objective function values are derived without any additional assumptions on the concave part. Numerical experiments performed on signal recovery problems show that the proposed algorithm is very competitive to notable algorithms in the recent literature, on both synthetic data and real data from two public datasets.

The rest of the paper is organized as follows. In Section~\ref{sec:preliminaries}, the preliminary materials used in this work are presented. In Section~\ref{sec:algorithm}, we introduce our algorithm and prove the subsequential and full sequential convergence of the proposed algorithm under suitable assumptions. Section~\ref{sec:numerical_result} presents the numerical results of the algorithm. Finally, the conclusions are given in Section~\ref{sec:conclusion}.

\section{Preliminaries}
\label{sec:preliminaries}

Let $f\colon \mathbb{R}^d\to \left[-\infty,+\infty\right]$. The \emph{domain} of $f$ is $\dom f :=\{x\in \mathbb{R}^d: f(x) <+\infty\}$ and the \emph{epigraph} of $f$ is $\epi f := \{(x,\rho)\in \mathbb{R}^d\times \mathbb{R}: f(x)\leq \rho\}$. The function $f$ is said to be \emph{proper} if $\dom f \neq \varnothing$ and it never takes the value $-\infty$, \emph{coercive} if $f(x)\to +\infty$ as $\|x\|\to +\infty$, \emph{lower semicontinuous} if its epigraph is a closed set, and \emph{convex} if its epigraph is a convex set. We say that $f$ is $\rho$-weakly convex if $\rho\in \mathbb{R}_+$ and $f +\frac{\rho}{2}\|\cdot\|^2$ is convex. 

Let $x\in \mathbb{R}^d$ with $|f(x)| <+\infty$. The \emph{regular subdifferential} of $f$ at $x$ is defined by
\begin{equation*}
\widehat{\partial} f(x) :=\left\{x^*\in \mathbb{R}^d:\; \liminf_{y\to x}\frac{f(y)-f(x)-\langle x^*,y-x\rangle}{\|y-x\|}\geq 0\right\}
\end{equation*}
and the \emph{limiting subdifferential} of $f$ at $x$ is defined by
\begin{equation*}
\partial f(x) :=\left\{x^*\in \mathbb{R}^d:\; \exists x_n\stackrel{f}{\to}x,\; x_n^*\to x^* \text{~~with~~} x_n^*\in \widehat{\partial} f(x_n)\right\},
\end{equation*}
where the notation $y\stackrel{f}{\to}x$ means $y\to x$ with $f(y)\to f(x)$. In the case where $\lvert f(x) \rvert =+\infty$, both regular subdifferential and limiting subdifferential of $f$ at $x$ are defined to be the empty set. The \emph{domain} of $\partial f$ is given by $\dom \partial f :=\{x\in \mathbb{R}^d: \partial f(x) \neq \varnothing\}$. It can be directly verified from the definition that the limiting subdifferential has the \emph{robustness property}
\begin{equation*}
\partial f(x) =\left\{x^*\in \mathbb{R}^d:\; \exists y_n \stackrel{f}{\to}x,\; y_n^*\to x^* \text{~~with~~} y_n^*\in \partial f(y_n)\right\}.
\end{equation*}


Given a function $f\colon \mathbb{R}^d \to [-\infty, +\infty]$, the \emph{Fenchel conjugate} $f^*\colon \mathbb{R}^d\to [-\infty, +\infty]$ of $f$ is defined by
\begin{align*}
f^*(v) =\sup_{x\in \mathbb{R}^d} \left(\langle v,x\rangle -f(x)\right).
\end{align*}
We recall some useful properties of the Fenchel conjugate in the following proposition.
\begin{proposition}
\label{p:conj}
Let $f\colon \mathbb{R}^d\to (-\infty, +\infty]$ be proper and let $x,v \in \mathbb{R}^d$. Then the following hold:
\begin{enumerate}
\item\label{p:conj_FY} 
$f^*$ is a proper lower semicontinuous convex function and $f(x) +f^*(v) \geq \langle x, v\rangle$. 
\item\label{p:conj_cvx} 
If $f$ is lower semicontinuous and convex, then $v\in \partial f(x)$ $\iff$ $f(x) +f^*(v) = \langle x, v\rangle$ $\iff$ $x \in \partial f^*(v)$. 
\end{enumerate}
\end{proposition}
\begin{proof}
\ref{p:conj_FY}: This follows from \cite[Propositions~13.13 and 13.15]{Bauschke2017_book}.

\ref{p:conj_cvx}: This follows from \cite[Theorem~16.29]{Bauschke2017_book}.
\end{proof}

Finally, we recall the concept of the well celebrated Kurdyka--Łojasiewicz (KL) property. A proper lower semicontinuous function $f \colon \mathbb{R}^d\to \left(-\infty, +\infty\right]$ satisfies the \emph{KL property} at $\overline{x} \in \dom \partial f$ if there are $\eta \in (0, +\infty]$, a neighborhood $V$ of $\overline{x}$, and a continuous concave function $\varphi: \left[0, \eta\right) \to \mathbb{R}_+$ such that $\varphi$ is continuously differentiable with $\varphi' > 0$ on $(0, \eta)$, $\varphi(0) = 0$, and, for all $x \in V$ with $
f(\overline{x}) < f(x) < f(\overline{x}) + \eta$, 
\begin{equation*}
\varphi'(f(x) -f(\overline{x})) \dist(0, \partial f(x)) \geq 1.   
\end{equation*}
The function $f$ is a \emph{KL function} if it satisfies the KL property at any point in $\dom \partial f$. If $f$ satisfies the KL property at $\overline{x} \in \dom \partial f$, in which the corresponding function $\varphi$ can be chosen as $\varphi(t) = c t ^{1 - \lambda}$ for some $c \in \mathbb{R}_{++}$ and $\lambda \in [0, 1)$, then $f$ is said to satisfy the \emph{KL property at $\overline{x}$ with exponent $\lambda$}. We say that $f$ is a \emph{KL function with exponent $\lambda$} if it is a KL function and has the same exponent $\lambda$ at any $x \in \dom \partial f$.

\section{Proposed algorithm and convergence analysis}
\label{sec:algorithm}

In this section, we present a splitting algorithm for solving problem~\eqref{eq:P} and perform the convergence analysis. We start by specifying the standing assumptions underlying our work.

\begin{assumption}[Standing assumptions]
\label{a:standing}
~
\begin{enumerate}
\item\label{a:standing_f} 
$f\colon \mathbb{R}^d\to \mathbb{R}$ is a differentiable $\rho$-weakly convex function with $\ell$-Lipschitz continuous gradient.
\item
$g\colon \mathbb{R}^d\to \mathbb{R}$ is a continuous convex function.
\item
$h\colon \mathbb{R}^d\to (-\infty, +\infty]$ is a lower semicontinuous function.
\end{enumerate} 
\end{assumption}

Regarding Assumption~\ref{a:standing}\ref{a:standing_f}, it is worth noting that $f$ is $\ell$-weakly convex as soon as it is a differentiable function with $\ell$-Lipschitz continuous gradient. We now introduce our algorithm, called \emph{backward-Douglas--Rachford algorithm} (BDR), to {\color{black}find a \emph{critical point} $x \in \mathbb{R}^d$ of \eqref{eq:P} in the sense that}
\begin{align*}
\color{black}0\in \nabla f(x) +\partial h(x) -\partial g(x).   
\end{align*}

\begin{tcolorbox}[
	left=0pt,right=0pt,top=0pt,bottom=0pt,
	colback=blue!10!white, colframe=blue!50!white,
  	boxrule=0.2pt,
  	breakable]
\begin{algorithm}[Backward-Douglas--Rachford algorithm (BDR)]
\label{algo:CDR}
\step{}
Choose initial points $y_0, z_0, w_0\in \mathbb{R}^d$ and set $n =0$. Let $\gamma \in \mathbb{R}_{++}$, $\tau\in \mathbb{R}_+$, and $\nu \in (0,2)$.

\step{}\label{step:main}
Compute
\begin{align*}
x_{n+1} &\in \argmin_{x\in \mathbb{R}^d} \left(f(x) +\frac{1}{2\gamma}\|x -y_n\|^2\right),\\
w_{n+1} &= \argmin_{w\in \mathbb{R}^d} \left(g^*(w) -\langle w, z_{n}\rangle +\frac{\tau}{2}\|w -w_n\|^2\right),\\
z_{n+1} &\in \argmin_{z\in \mathbb{R}^d} \left(h(z) +\frac{1}{2\gamma}\|z -(2x_{n+1} -y_n +\gamma w_{n+1})\|^2\right),\\
y_{n+1} &= y_n +\nu(z_{n+1} -x_{n+1}).
\end{align*}

\step{}
If a termination criterion is not met, set $n =n+1$ and go to Step~\ref{step:main}.
\end{algorithm}
\end{tcolorbox}

\begin{remark}\label{remark:comment} 
Some comments on Algorithm~\ref{algo:CDR} are as follows.
\begin{enumerate}
\item \label{remark:comment_i}
If $\tau >0$, then the updating schemes in Step~\ref{step:main} can be written as
\begin{align*}
\begin{cases}
x_{n+1} &\in \prox_{\gamma f}(y_n),\\
w_{n+1} &= \prox_{\frac{1}{\tau}g^*}\left(w_n +\frac{1}{\tau}z_n\right),\\
z_{n+1} &\in \prox_{\gamma h}(2x_{n+1} -y_n +\gamma w_{n+1}),\\
y_{n+1} &= y_n +\nu(z_{n+1} -x_{n+1}),  
\end{cases}
\end{align*}
where $\prox_{\gamma \phi}$ is the proximity operator of a proper function $\phi\colon \mathbb{R}^d\to \left(-\infty,+\infty\right]$ with parameter $\gamma\in \mathbb{R}_{++}$ defined by
\begin{align*}
\prox_{\gamma \phi}(x) =\argmin_{y\in \mathbb{R}^d}\left(\phi(y) +\frac{1}{2\gamma}\|y -x\|^2\right).
\end{align*}
Furthermore, by Moreau's decomposition \cite[Theorem~14.3(ii)]{Bauschke2017_book},
\begin{align}\label{eq:Moreau}
w_{n+1} = \frac{1}{\tau}(\tau w_n +z_n) -\frac{1}{\tau}\prox_{\tau g}\left(\tau w_n +z_n\right).  
\end{align}
The $w$-update generates a sequence along the convex conjugate of $g$, which can be interpreted as a dual sequence. Therefore, the BDR can also be viewed as a primal-dual type algorithm.

{\color{black}As this algorithm is a proximal type one, it is a standard assumption that $f$, $g$, and $h$ are prox-friendly functions, so that the updating step can be done efficiently. In the case that they are not prox-friendly, bundle methods can be a good alternative since they might not require this prox-friendly assumption; see, e.g., \cite{bundle_DC}.}

\item 
When $g\equiv 0$, Algorithm~\ref{algo:CDR} becomes the relaxed DR algorithm which reduces to the classical DR algorithm if $\nu =1$. 
\end{enumerate}
\end{remark}

The following lemma will be utilized multiple times in our analysis. 

\begin{lemma}
\label{l:Rela}
Suppose that $f$ is differentiable with $\ell$-Lipschitz continuous gradient. Let $(x_n, y_n, z_n, w_n)_{n\in \mathbb{N}^*}$ be a sequence generated by Algorithm~\ref{algo:CDR}. Then, for all $n\in \mathbb{N}$, the following hold:
\begin{enumerate}
\item\label{l:Rela_y}
$y_n = x_{n+1} +\gamma \nabla f(x_{n+1})$.
\item\label{l:Rela_z} 
$z_n \in \partial g^*(w_{n+1}) +\tau(w_{n+1} -w_n)$.
\item\label{l:Rela_w}
$w_{n+1} \in \partial h(z_{n+1}) -\frac{1}{\gamma}(x_{n+1} -z_{n+1}) -\frac{1}{\gamma}(x_{n+1} -y_n)$.
\item\label{l:Rela_yy}
$\|y_{n+1} -y_n\|\leq (1 +\gamma\ell)\|x_{n+2} -x_{n+1}\|$.
\item\label{l:Rela_zz} 
$\|z_{n+2} -z_{n+1}\|\leq \frac{1 +\gamma\ell}{\nu}\|x_{n+3} -x_{n+2}\| +\left(1 +\frac{1 +\gamma\ell}{\nu}\right)\|x_{n+2} -x_{n+1}\|$.
\end{enumerate}    
\end{lemma}
\begin{proof}
Let $n\in \mathbb{N}$. The optimality conditions for the subproblems in Step~\ref{step:main} of Algorithm~\ref{algo:CDR} are 
\begin{align*}
0 &= \nabla f(x_{n+1}) +\frac{1}{\gamma}(x_{n+1} -y_n),\\
0 &\in {\color{black}\partial} g^*(w_{n+1}) -z_{n} +\tau(w_{n+1} -w_n),\\
0 &\in {\color{black}\partial} h(z_{n+1}) +\frac{1}{\gamma}(z_{n+1} -(2x_{n+1} -y_n +\gamma w_{n+1}))\\
&= {\color{black}\partial} h(z_{n+1}) -\frac{1}{\gamma}(x_{n+1} -z_{n+1}) -\frac{1}{\gamma}(x_{n+1} -y_n) -w_{n+1},
\end{align*}
which imply \ref{l:Rela_y}, \ref{l:Rela_z}, and \ref{l:Rela_w}.

Now, it follows from \ref{l:Rela_y} and the Lipschitz continuity of $\nabla f$ that
\begin{align*}
\|y_{n+1} -y_n\|\leq \|x_{n+2} -x_{n+1}\| +\gamma\|\nabla f(x_{n+2}) -\nabla f(x_{n+1}))\|\leq (1 +\gamma\ell)\|x_{n+2} -x_{n+1}\|,
\end{align*}
and we obtain \ref{l:Rela_yy}. Next, we derive from the updating step of $y_{n+1}$ that $\nu(z_{n+2} -z_{n+1}) =(y_{n+2} -y_{n+1}) -(y_{n+1} -y_n) +\nu(x_{n+2} -x_{n+1})$. By combining with \ref{l:Rela_yy}, 
\begin{align*}
\nu\|z_{n+2} -z_{n+1}\| &\leq \|y_{n+2} -y_{n+1}\| +\|y_{n+1} -y_n\| +\nu\|x_{n+2} -x_{n+1}\| \\
&\leq (1 +\gamma\ell)\|x_{n+3} -x_{n+2}\| +(1 +\gamma\ell)\|x_{n+2} -x_{n+1}\| +\nu\|x_{n+2} -x_{n+1}\|,
\end{align*}
which implies \ref{l:Rela_zz}.
\end{proof}

The convergence analysis of the BDR will be based on the Lyapunov function
\begin{align*}
\mathcal{F}(x,y,z,w) =f(x) +h(z) +g^*(w) -\langle w, z\rangle +\frac{1}{2\gamma}\|x -y\|^2 -\frac{1}{2\gamma}\|y -z\|^2 +\frac{1 -\nu}{\gamma}\|x -z\|^2.
\end{align*}

\begin{theorem}[Subsequential convergence]
\label{theorem:sub_convergence} 
Suppose that Assumption~\ref{a:standing} holds and that 
\begin{align*}
\gamma < \overline{\gamma}:=\frac{-\nu\rho +\sqrt{\nu^2\rho^2 +8(2 -\nu)\ell^2}}{4\ell^2},
\end{align*}
where we adopt the convention that the right-hand side becomes $+\infty$ when $\ell =0$. Let $(x_n, y_n, z_n, w_n)_{n\in \mathbb{N}^*}$ be a sequence generated by Algorithm~\ref{algo:CDR}. Then the following hold: 
\begin{enumerate}
\item\label{sub_convergence_descend} 
For all $n\in \mathbb{N}^*$, 
\begin{align}\label{eq:decrease}
\mathcal{F}(x_{n+1},y_{n+1},z_{n+1},w_{n+1}) 
+\frac{\delta}{2}\|x_{n+1} -x_n\|^2 +\frac{\tau}{2}\|w_{n+1} -w_n\|^2 \leq \mathcal{F}(x_{n},y_{n},z_{n},w_{n}),
\end{align}
where $\delta :=\frac{2 -\nu -\nu\rho\gamma -2\ell^2\gamma^2}{\nu\gamma}$. Consequently, the sequence $(\mathcal{F}(x_{n},y_{n},z_{n},w_{n}))_{n\in\mathbb{N}^*}$ is nonincreasing.

\item\label{bounded_of_sequence} 
Suppose that {\color{black}$\mathbf{F}$} is coercive. Then the sequence $(x_n,y_n,z_n,w_n)_{n\in \mathbb{N}^*}$ is bounded and $\|x_{n+1} -x_n\| \to 0$, $\|y_{n+1} -y_n\| \to 0$, $\|z_{n+1} -z_n\| \to 0$,  and $\tau\|w_{n+1} -w_n\| \to 0$ as $n \to +\infty$. Moreover, for every cluster point  $(\overline{x},\overline{y},\overline{z},\overline{w})$, it holds that  
\begin{align*}
&\overline{x} =\overline{z}, \\ 
&0\in \nabla f(\overline{z}) +\partial h(\overline{z}) -\partial g(\overline{z}), \text{~and} \\ 
&\lim_{n\to+\infty} {\color{black} \mathbf{F}(z_n)} =\lim_{n\to+\infty} \mathcal{F}(x_n,y_n,z_n,w_n) =\mathcal{F}(\overline{x},\overline{y},\overline{z},\overline{w}) ={\color{black}\mathbf{F}(\overline{z})}.
\end{align*}
\end{enumerate}
\end{theorem}
\begin{proof}
\ref{sub_convergence_descend}: 
%
Let $n\in \mathbb{N}$. First, we have that
\begin{align}\label{eq:descend_x}
&\mathcal{F}(x_{n+1},y_{n+1},z_{n+1},w_{n+1}) -\mathcal{F}(x_{n+1},y_{n},z_{n+1},w_{n+1}) \notag \\ 
&= \frac{1}{2\gamma}\|x_{n+1} -y_{n+1}\|^2 -\frac{1}{2\gamma}\|y_{n+1} -z_{n+1}\|^2 -\frac{1}{2\gamma}\|x_{n+1} -y_{n}\|^2 +\frac{1}{2\gamma}\|y_{n} -z_{n+1}\|^2 \notag \\ 
&= \frac{1}{\gamma}\langle y_{n+1} -y_n, z_{n+1} -x_{n+1} \rangle \notag \\
&= \frac{\nu}{\gamma}\|x_{n+1} -z_{n+1}\|^2,
\end{align}
where the last equality follows from the updating step of $y_{n+1}$. Next, we observe that
\begin{align*}
\mathcal{F}(x,y,z,w) =f(x) +h(z) +g^*(w) -\langle x, w\rangle +\frac{1}{2\gamma}\|2x -y +\gamma w -z\|^2 -\frac{1}{2\gamma}\|x -y +\gamma w\|^2 -\frac{\nu}{\gamma}\|x -z\|^2,
\end{align*}
which comes from the definition of $\mathcal{F}(x,y,z,w)$ and the fact that 
\begin{align*}
\|2x -y +\gamma w -z\|^2 -\|x -y +\gamma w\|^2 &= \|(x -z) +(x -y +\gamma w)\|^2 -\|x -y +\gamma w\|^2 \\
&= \|x -z\|^2 +2\langle x -z, x -y +\gamma w\rangle \\
&= \|x -z\|^2 +2\langle x -z, x -y\rangle +2\gamma\langle x -z, w\rangle \\
&= 2\|x -z\|^2 +\|x -y\|^2 -\|y -z\|^2 +2\gamma\langle x -z, w\rangle.
\end{align*}
Therefore,
\begin{align}\label{eq:descend_z}
&\mathcal{F}(x_{n+1},y_{n},z_{n+1},w_{n+1}) -\mathcal{F}(x_{n+1},y_{n},z_{n},w_{n+1}) \notag \\
&= h(z_{n+1}) +\frac{1}{2\gamma}\|2x_{n+1} -y_{n} +\gamma w_{n+1} -z_{n+1}\| -\frac{\nu}{\gamma}\|x_{n+1} -z_{n+1}\|^2 \notag \\
&\quad -h(z_{n}) -\frac{1}{2\gamma}\|2x_{n+1} -y_{n} +\gamma w_{n+1} -z_{n}\|^2 +\frac{\nu}{\gamma}\|x_{n+1} -z_{n}\|^2 \notag \\
&\leq -\frac{\nu}{\gamma}\|x_{n+1} -z_{n+1}\|^2 +\frac{\nu}{\gamma}\|x_{n+1} -z_n\|^2, 
\end{align}
where the inequality is obtained from the definition of $z_{n+1}$ as a minimizer. We also see that $x_{n+1}$ is a minimizer of $f +\frac{1}{2\gamma}\|\cdot -y_n\|^2$, which results in
\begin{align}\label{eq:descend_y}
&\mathcal{F}(x_{n+1},y_{n},z_{n},w_{n+1}) -\mathcal{F}(x_{n},y_{n},z_{n},w_{n+1}) \notag \\
&= f(x_{n+1}) +\frac{1}{2\gamma}\|x_{n+1} -y_n\|^2 -f(x_n) -\frac{1}{2\gamma}\|x_n -y_n\|^2 +\frac{1 -\nu}{\gamma}\|x_{n+1} -z_n\|^2 -\frac{1 -\nu}{\gamma}\|x_n -z_n\|^2 \notag \\
&\leq -\frac{1}{2}\left(\frac{1}{\gamma} -\rho\right)\|x_{n+1} -x_n\|^2 +\frac{1 -\nu}{\gamma}\|x_{n+1} -z_n\|^2 -\frac{1 -\nu}{\gamma}\|x_n -z_n\|^2 .
\end{align}
Furthermore, it follows from the definition of $w_{n+1}$ that
\begin{align}\label{eq:descend_p}
\mathcal{F}(x_{n},y_{n},z_{n},w_{n+1}) -\mathcal{F}(x_{n},y_{n},z_{n},w_{n}) =g^*(w_{n+1}) -\langle w_{n+1}, z_n\rangle -g^*(w_{n}) +\langle w_{n}, z_n\rangle \leq -\frac{\tau}{2}\|w_{n+1} -w_n\|^2.
\end{align}
Summing up \eqref{eq:descend_x}, \eqref{eq:descend_z}, \eqref{eq:descend_y}, and \eqref{eq:descend_p}, and then noting $x_{n+1} -z_n =(x_{n+1} -x_n) +(x_n -z_n)$, we obtain
\begin{align}\label{eq:descend_F-}
&\mathcal{F}(x_{n+1},y_{n+1},z_{n+1},w_{n+1}) -\mathcal{F}(x_{n},y_{n},z_{n},w_{n}) \notag \\
&\leq \left(-\frac{1}{2\gamma}+\frac{\rho}{2}\right)\|x_{n+1} -x_n\|^2 +\frac{1}{\gamma}\|x_{n+1} -z_n\|^2 +\frac{\nu-1}{\gamma}\|x_n-z_n\|^2 -\frac{\tau}{2} \|w_{n+1} -w_n\|^2 \notag \\
&= \left(\frac{1}{2\gamma} +\frac{\rho}{2}\right)\|x_{n+1} -x_n\|^2 +\frac{2}{\gamma}\langle x_{n+1} -x_n, x_n -z_n\rangle +\frac{\nu}{\gamma}\|x_n -z_n\|^2 -\frac{\tau}{2}\|w_{n+1} -w_n\|^2.
\end{align}

Now, let $n\in \mathbb{N}^*$. It follows from Lemma~\ref{l:Rela}\ref{l:Rela_y} that
\begin{align*}
&\nu(x_n -z_n) =-(y_n -y_{n-1}) =-(x_{n+1} -x_n) -\gamma(\nabla f(x_{n+1}) -\nabla f(x_n)),
\end{align*}
which leads to
\begin{align*}
\frac{\nu}{\gamma}\|x_n -z_n\|^2 &= \frac{1}{\nu\gamma}\|(x_{n+1} -x_n) +\gamma(\nabla f(x_{n+1}) -\nabla f(x_n))\|^2 \notag \\
&= \frac{1}{\nu\gamma}\|x_{n+1} -x_n\|^2 +\frac{\gamma}{\nu}\|\nabla f(x_{n+1}) -\nabla f(x_n)\|^2 +\frac{2}{\nu}\langle x_{n+1} -x_n, \nabla f(x_{n+1}) -\nabla f(x_n) \rangle,
\end{align*}
and also
\begin{align*}
\frac{2}{\gamma}\langle x_{n+1} -x_n, x_n -z_n\rangle &= -\frac{2}{\nu\gamma} \langle x_{n+1} -x_n, (x_{n+1} -x_n) +\gamma(\nabla f(x_{n+1}) -\nabla f(x_n))\rangle \notag \\
&= -\frac{2}{\nu\gamma}\|x_{n+1} -x_n\|^2 -\frac{2}{\nu}\langle x_{n+1} -x_n, \nabla f(x_{n+1}) -\nabla f(x_n)\rangle.
\end{align*}
Combining with \eqref{eq:descend_F-} and using the Lipschitz continuity of $\nabla f$, we obtain that
\begin{align}\label{eq:descend_F}
&\mathcal{F}(x_{n+1},y_{n+1},z_{n+1},w_{n+1}) -\mathcal{F}(x_{n},y_{n},z_{n},w_{n}) \notag \\
&\leq \left(\frac{1}{2\gamma} +\frac{\rho}{2} -\frac{1}{\nu\gamma}\right)\|x_{n+1} -x_n\|^2 +\frac{\gamma}{\nu}\|\nabla f(x_{n+1} -\nabla f(x_n)\|^2 -\frac{\tau}{2}\|w_{n+1} -w_n\|^2 \notag \\
&\leq \left(\frac{1}{2\gamma} +\frac{\rho}{2} -\frac{1}{\nu\gamma} +\frac{\ell^2\gamma}{\nu}\right)\|x_{n+1} -x_n\|^2 -\frac{\tau}{2}\|w_{n+1} -w_n\|^2,
\end{align}
which implies \eqref{eq:decrease}. Now, by the choice of $\gamma$, $\tau$, and $\nu$, it holds that {\color{black} $\delta =\frac{2 -\nu -\nu\rho\gamma -2\ell^2\gamma^2}{\nu\gamma}>0$ } and $\tau \geq 0$, hence $(\mathcal{F}(x_{n},y_{n},z_{n},w_{n}))_{n\in\mathbb{N}^*}$ is nonincreasing.


\ref{bounded_of_sequence}: Let $n\in \mathbb{N}^*$. Since $\nabla f$ is $\ell$-Lipschitz continuous, it follows from \cite[Lemma~1.2.3]{Nesterov2018} that
\begin{align}\label{eq:fxn}
f(z_n) +\langle \nabla f(x_n), x_n -z_n\rangle -\frac{\ell}{2}\|x_n -z_n\|^2 \leq f(x_n) \leq f(z_n) +\langle \nabla f(x_n), x_n -z_n\rangle +\frac{\ell}{2}\|x_n -z_n\|^2.    
\end{align}
In addition, we derive from Lemma~\ref{l:Rela}\ref{l:Rela_y} and the updating step of $y_{n+1}$ that
\begin{align*}
\gamma \nabla f(x_n) =y_{n-1} -x_n = y_n -\nu(z_n -x_n) -x_n =(y_n -z_n) -(1 -\nu)(x_n -z_n),
\end{align*}
and so
\begin{align}\label{eq:nabla fxn}
\langle \nabla f(x_n), x_n -z_n\rangle &= \frac{1}{\gamma}\langle y_n -z_n, x_n -z_n\rangle -\frac{1 -\nu}{\gamma}\|x_n -z_n\|^2 \notag \\
&= \frac{1}{2\gamma}\|y_n -z_n\|^2 +\frac{1}{2\gamma}\|x_n -z_n\|^2 -\frac{1}{2\gamma}\|x_n -y_n\|^2 -\frac{1 -\nu}{\gamma}\|x_n -z_n\|^2.
\end{align}

Next, by Proposition~\ref{p:conj}\ref{p:conj_FY},
\begin{align}\label{eq:g*wn>}
g^*(w_n) -\langle w_n, z_n\rangle \geq -g(z_n).    
\end{align}
According to Lemma~\ref{l:Rela}\ref{l:Rela_z}, $z_{n-1} -\tau(w_n -w_{n-1}) \in \partial g^*(w_n)$. By the convexity of $g$ and Proposition~\ref{p:conj}\ref{p:conj_cvx},
\begin{align}\label{eq:g*wn}
g^*(w_n) -\langle w_n, z_n\rangle 
&= \langle w_n, z_{n-1} -\tau(w_n -w_{n-1})\rangle -g(z_{n-1} -\tau(w_n -w_{n-1})) -\langle w_n, z_n\rangle \notag \\
&= -\langle w_n, (z_n -z_{n-1}) +\tau(w_n -w_{n-1})\rangle -g(z_{n-1} -\tau(w_n -w_{n-1})).
\end{align}
Let $t_n\in \partial g(z_n)$. Again using the convexity of $g$, we have that
\begin{align*}
g(z_{n-1} -\tau(w_n -w_{n-1})) \geq g(z_n) -\langle t_n, (z_n -z_{n-1}) +\tau(w_n -w_{n-1})\rangle,
\end{align*}
which combined with \eqref{eq:g*wn} yields
\begin{align}\label{eq:g*wn<}
g^*(w_n) -\langle w_n, z_n\rangle \leq -g(z_n) +\langle t_n -w_n, (z_n -z_{n-1}) +\tau(w_n -w_{n-1})\rangle.
\end{align}
Combining \eqref{eq:fxn}, \eqref{eq:nabla fxn}, \eqref{eq:g*wn>}, and \eqref{eq:g*wn<}, we obtain that
\begin{align}\label{eq:Fn}
\mathcal{F}(x_n,y_n,z_n,w_n)\geq f(z_n) +h(z_n) -g(z_n) +\left( \frac{1}{2\gamma} -\frac{\ell}{2}\right)\|x_n -z_n\|^2 ={\color{black}\mathbf{F}(z_n)} +\left( \frac{1}{2\gamma} -\frac{\ell}{2}\right)\|x_n -z_n\|^2   
\end{align}
and that
\begin{align}\label{eq:Fn'}
\mathcal{F}(x_n,y_n,z_n,w_n)\leq {\color{black} \mathbf{F}(z_n)} +\langle t_n -w_n, (z_n -z_{n-1}) +\tau(w_n -w_{n-1})\rangle +\left(\frac{1}{2\gamma} +\frac{\ell}{2}\right)\|x_n -z_n\|^2.    
\end{align}

We now observe that
\begin{align}\label{eq:gamma}
\frac{1}{2\gamma} -\frac{\ell}{2} >0.    
\end{align}
Indeed, if $\ell =0$, then $\frac{1}{2\gamma} -\frac{\ell}{2} =\frac{1}{2\gamma} >0$. If $\ell >0$, then
\begin{align*}
\gamma < \frac{-\nu\rho +\sqrt{\nu^2\rho^2 +8(2 -\nu)\ell^2}}{4\ell^2} \leq \frac{-\nu\rho +(\sqrt{\nu^2\rho^2} +\sqrt{8(2 -\nu)\ell^2})}{4\ell^2} <\frac{-\nu\rho +(\nu\rho +4\ell)}{4\ell^2} =\frac{1}{\ell}.    
\end{align*}
So, we always have \eqref{eq:gamma}. Moreover, {\color{black}$\mathbf{F}$} is bounded below since it is a proper lower semicontinuous coercive function \cite[Theorem~1.9]{RW98}. We then derive from \eqref{eq:Fn} that the sequence $(\mathcal{F}(x_n,y_n,z_n,w_n))_{n\in \mathbb{N}^*}$ is bounded below. According to \ref{sub_convergence_descend}, $(\mathcal{F}(x_n,y_n,z_n,w_n))_{n\in \mathbb{N}^*}$ is nonincreasing, so it is convergent.

Telescoping \eqref{eq:decrease} yields
\begin{align*}
\frac{\delta}{2}\sum_{n=1}^{+\infty} \|x_{n+1} -x_n\|^2 +\frac{\tau}{2}\sum_{n=1}^{+\infty} \|w_{n+1} -w_n\|^2 &\leq \sum_{n=1}^{+\infty} (\mathcal{F}(x_n,y_n,z_n,w_n) -\mathcal{F}(x_{n+1},y_{n+1},z_{n+1},w_{n+1})) \\
&= \mathcal{F}(x_1,w_1,z_1,y_1) -\lim_{n\to +\infty}\mathcal{F}(x_n,y_n,z_n,w_n) < +\infty,
\end{align*}
which implies that $\sum_{n=1}^{+\infty} \|x_{n+1} -x_n\|^2< +\infty$ and $\sum_{n=1}^{+\infty} \tau\|w_{n+1} -w_n\|^2< +\infty$, hence $\|x_{n+1} -x_n\|\to 0$ and $\tau \|w_{n+1}-w_n\|\to 0$ as $n\to +\infty$. Combining with Lemma~\ref{l:Rela}\ref{l:Rela_yy}\&\ref{l:Rela_zz}, we also obtain that $\|y_{n+1} -y_n\|\to 0$ and $\|z_{n+1}-z_{n}\|\to 0$ as $n\to +\infty$.

Now, we derive from \ref{sub_convergence_descend}, \eqref{eq:Fn}, and the below boundedness of {\color{black}$\mathbf{F}$} that 
{\color{black}$(\mathbf{F}(z_n))_{n\in \mathbb{N}^*}$} and $(\|x_n -z_n\|)_{n\in \mathbb{N}^*}$ are bounded. 
By the coercivity of {\color{black}$\mathbf{F}$}, we obtain the boundedness of $(z_n)_{n\in \mathbb{N}^*}$ and then that of $(x_n)_{n\in \mathbb{N}^*}$. Combining with Lemma~\ref{l:Rela}\ref{l:Rela_y}, we deduce that $(y_n)_{n\in \mathbb{N}^*}$ is also bounded. 
In view of Lemma~\ref{l:Rela}\ref{l:Rela_z} and Proposition~\ref{p:conj}\ref{p:conj_cvx}, $w_{n+1}\in \partial g(z_n -\tau(w_{n+1} -w_n))$. Since $t_n\in \partial g(z_n)$, $(z_n)_{n\in \mathbb{N}^*}$ is bounded, and $\tau\|w_{n+1} -w_n\|\to 0$ as $n\to +\infty$, we have from \cite[Proposition~16.20]{Bauschke2017_book} that
\begin{align}\label{eq:wntn bounded}
\text{$(w_n)_{n\in \mathbb{N}^*}$ and $(t_n)_{n\in \mathbb{N}^*}$ are bounded}.
\end{align} 


Let $(\overline{x},\overline{y},\overline{z},\overline{w})$ be any cluster point of the sequence $(x_n,y_n,z_n,w_n)_{n\in \mathbb{N}^*}$. Then there exists a subsequence $(x_{k_n},y_{k_n},z_{k_n},w_{k_n})_{n\in \mathbb{N}}$ that converges to $(\overline{x},\overline{y},\overline{z},\overline{w})$. Since $x_{n+1} -x_n \to 0$, $x_{n+1} -z_{n+1} =-\frac{1}{v}(y_{n+1} -y_n) \to 0$, $z_{n+1} -z_n \to 0$, and $\tau(w_{n+1} -w_n) \to 0$ as $n \to +\infty$, we have that $\overline{x} =\overline{z}$, that  
\begin{align}\label{eq:asymptotic xyz}
\lim_{n\to +\infty} (x_{k_n-1},y_{k_n-1},z_{k_n-1}) =\lim_{n\to +\infty} (x_{k_n},y_{k_n},z_{k_n}) =(\overline{x},\overline{y},\overline{z}),
\end{align}
and that
\begin{align}\label{eq:asymptotic w}
\lim_{n\to +\infty} w_{k_n-1} =\lim_{n\to +\infty} w_{k_n} =\overline{w} \text{~~if~} \tau >0.
\end{align}  
We have from the updating step of $z_{n+1}$ that
\begin{align*}
h(z_{k_n}) +\frac{1}{2\gamma}\|z_{k_n} -(2x_{k_n} -y_{k_n-1} +\gamma w_{k_n})\|^2 \leq  h(\overline{z}) +\frac{1}{2\gamma}\|\overline{z} -(2x_{k_n} -y_{k_n-1} +\gamma w_{k_n})\|^2
\end{align*}
and from the updating step of $w_{n+1}$ that
\begin{align*}
g^*(w_{k_n}) -\langle w_{k_n}, z_{k_n-1}\rangle +\frac{\tau}{2}\|w_{k_n} -w_{k_n-1}\|^2 \leq g^*(\overline{w}) -\langle \overline{w}, z_{k_n-1}\rangle +\frac{\tau}{2}\|\overline{w} -w_{k_n-1}\|^2.
\end{align*}
Taking the limit and using \eqref{eq:asymptotic xyz} and \eqref{eq:asymptotic w}, we derive that $\limsup_{n\to +\infty} h(z_{k_n}) \leq h(\overline{z})$ and $\limsup_{n\to +\infty} g^*(w_{k_n}) \leq g^*(\overline{w})$. Together with the lower semicontinuity of $h$ and $g^*$, this yields  
\begin{align}\label{eq:lim_h_g_star}
\lim_{n\to +\infty} h(z_{k_n}) =h(\overline{z}) \text{~and~} \lim_{n \to +\infty} g^*(w_{k_n}) =g^*(\overline{w}).   
\end{align}
By Lemma~\ref{l:Rela}\ref{l:Rela_y}--\ref{l:Rela_w},
\begin{align*}
w_{k_n} +\frac{1}{\gamma}(x_{k_n} -z_{k_n})\in \nabla f(x_{k_n}) +\partial h(z_{k_n}) \text{~~and~~} z_{k_n-1} -\tau(w_{k_n} -w_{k_n-1})\in  \partial g^*(w_{k_n}).    
\end{align*}
Passing to the limit, we obtain that $\overline{w}\in \nabla f(\overline{x}) +\partial h(\overline{z}) =\nabla f(\overline{z}) +\partial h(\overline{z})$ and $\overline{z}\in \partial g^*(\overline{w})$. The latter implies 
\begin{align}\label{eq:g*g}
\overline{w}\in \partial g(\overline{z}) \text{~~and~~} g^*(\overline{w}) -\langle \overline{w}, \overline{z}\rangle =-g(\overline{z})
\end{align}
due to the convexity of $g$ and Proposition~\ref{p:conj}\ref{p:conj_cvx}. Therefore, $0\in \nabla f(\overline{z}) +\partial h(\overline{z}) -\partial g(\overline{z})$. 

Next, let us note that 
\begin{align*}
\mathcal{F}(x_{k_n},y_{k_n},z_{k_n},w_{k_n}) &=f(x_{k_n}) +h(z_{k_n}) +g^*(w_{k_n}) -\langle w_{k_n}, z_{k_n}\rangle \\
&\quad +\frac{1}{2\gamma}\|x_{k_n}  -y_{k_n}\|^2 -\frac{1}{2\gamma}\|y_{k_n} -z_{k_n}\|^2 +\frac{1 -\nu}{\gamma}\|x_{k_n} -z_{k_n}\|^2.
\end{align*}
By the continuity of $f$ and \eqref{eq:lim_h_g_star},
\begin{align*}
\lim_{n\to +\infty} \mathcal{F}(x_{k_n},y_{k_n},z_{k_n},w_{k_n}) =f(\overline{z}) +h(\overline{z}) +g^*(\overline{w}) -\langle \overline{w}, \overline{z}\rangle =\mathcal{F}(\overline{x},\overline{y},\overline{z},\overline{w}).
\end{align*}
Combining with \eqref{eq:g*g} and the convergence of the whole sequence $(\mathcal{F}(x_n,y_n,z_n,w_n))_{n\in \mathbb{N}^*}$, we deduce that
\begin{align*}
\lim_{n\to +\infty} \mathcal{F}(x_n,y_n,z_n,w_n) = \mathcal{F}(\overline{x},\overline{y},\overline{z},\overline{w}) ={\color{black}\mathbf{F}(\overline{z})}.
\end{align*}
Finally, using \eqref{eq:Fn}, \eqref{eq:Fn'}, and \eqref{eq:wntn bounded} as well as noting that $x_n -z_n =-\frac{1}{\nu}(y_n -y_{n-1})\to 0$, $z_n -z_{n-1}\to 0$, and $\tau(w_n -w_{n-1})\to 0$ as $n\to +\infty$, we have $\lim_{n\to +\infty} \mathcal{F}(x_n,y_n,z_n,w_n) =\lim_{n\to +\infty} {\color{black}\mathbf{F}(z_n)}$, which completes the proof.
\end{proof}

{\color{black} We want to note that in the case where the constants $\ell$ and $\rho$ are unknown, the bound $\overline{\gamma}$ should be chosen sufficiently small to have guaranteed convergence. One strategy is that this bound can be initialized as a large value, and decreased gradually during the iterative process, as shown in \cite[Remark 4]{Li2015DR}.} Next, we prove the full sequential convergence and deduce the convergence rates of both the sequence generated by the BDR and the corresponding objective function values.

\begin{theorem}[Full sequential convergence]\label{t:full_sequence_convergence} Suppose that Assumption~\ref{a:standing} holds, that {\color{black}$\mathbf{F}$} is coercive, and that 
\begin{align*}
   0 <\gamma < \overline{\gamma}:= \frac{-\nu\rho +\sqrt{\nu^2\rho^2 -8(\nu -2)\ell^2}}{4\ell^2}.
\end{align*}
Let $(x_n,y_n,z_n,w_n)_{n\in \mathbb{N}^*}$ be the sequence generated by Algorithm~\ref{algo:CDR} with $\tau\in (0, +\infty)$.
\begin{enumerate}
\item\label{t:full_cvg} 
Suppose that $\mathcal{F}$ is a KL function. Then the sequence $(x_{n},y_{n},z_{n},w_{n})_{n\in\mathbb{N}^*}$ converges to $(x^*,y^*,z^*,w^*)$, where $x^* =z^*$, $0\in \nabla f(z^*) +\partial h(z^*) -\partial g(z^*)$, and 
\begin{align*}
\sum_{n=0}^{+\infty} \|(x_{n+1},y_{n+1},z_{n+1},w_{n+1})-(x_{n},y_{n},z_{n},w_{n})\| < +\infty.
\end{align*}
\item\label{t:full_rate} 
Suppose that $\mathcal{F}$ is a KL function with exponent $\theta \in [0,1)$. Then the following hold:
\begin{enumerate}
\item\label{t:full_rate_finite}
If $\theta =0$, then $(x_n,y_n,z_n,w_n)_{n\in\mathbb{N}^*}$ converges to $(x^*,y^*,z^*,w^*)$ in a finite number of steps. 
\item\label{t:full_rate_linear}
If $\theta\in (0,\frac{1}{2}]$, then there exist $\Gamma\in \mathbb{R}_{++}$ and $\zeta\in \left(0,1\right)$ such that, for all $n \in \mathbb{N}^*$, 
\begin{align*}
\|(x_n,y_n,z_n, w_n)-(x^*,y^*,z^*,w^*)\|\leq  \Gamma \zeta^{\frac{n}{2}}~\text{and}~|{\color{black}\mathbf{F}(z_n)-\mathbf{F}(z^*)}|\leq\Gamma \zeta^{\frac{n}{2}}.
\end{align*}
\item\label{t:full_rate_sublinear}
If $\theta\in (\frac{1}{2},1)$, then there exists $\Gamma\in \mathbb{R}_{++}$ such that, for all $n \in \mathbb{N}^*$,
\begin{align*}
\|(x_n,y_n,z_n, w_n)-(x^*,y^*,z^*,w^*)\| \leq \Gamma n^{-\frac{1-\theta}{2\theta-1}}~\text{and}~|{\color{black}\mathbf{F}(z_n)-\mathbf{F}(z^*)}|\leq\Gamma n^{-\frac{1-\theta}{2\theta-1}}.
\end{align*}
\end{enumerate}
\end{enumerate}
\end{theorem}

\begin{proof}
For each $n\in\mathbb{N}$, set $v_n =(x_n, y_n, z_n, w_n)$ and $\Delta_n =\|x_{n+1} -x_n\| +\tau\|w_{n+1} -w_n\|$. Let $n\in \mathbb{N}^*$. We derive from Theorem~\ref{theorem:sub_convergence} that
\begin{align}\label{eq:F_descend_global}
\mathcal{F}(v_{n+1}) +\frac{\delta}{2}\|x_{n+1} -x_n\|^2 +\frac{\tau}{2}\|w_{n+1} -w_n\|^2 \leq \mathcal{F}(v_{n}) \text{~~with~} \delta :=\frac{2 -\nu -\nu\rho\gamma -2\ell^2\gamma^2}{\nu\gamma} >0,
\end{align}
that $v_{n+1} -v_n\to 0$ as $n \to +\infty$, that $(v_n)_{n\in \mathbb{N}}$ is bounded, and for every cluster point $\overline{v} =(\overline{x},\overline{y},\overline{z},\overline{w})$ of $(v_n)_{n\in \mathbb{N}}$, one has $\overline{x} =\overline{z}$, and $\mathcal{F}(v_n) \to \mathcal{F}(\overline{v})={\color{black}\mathbf{F}(\overline{z})}$ as $n \to +\infty$. Next, we observe that
\begin{align*}
\Delta_n^2 \leq \left(\frac{2}{\delta} +2\tau\right)\left(\frac{\delta}{2}\|x_{n+1} -x_n\|^2 +\frac{\tau}{2}\|w_{n+1} -w_n\|^2\right).
\end{align*}
By combining with \eqref{eq:F_descend_global},
\begin{align*}
\mathcal{F}(v_{n+1}) +\frac{\delta}{2 +2\delta\tau}\Delta_n^2 \leq \mathcal{F}(v_n).
\end{align*}

We now see that $\partial \mathcal{F}(v_n) =(\partial^x \mathcal{F}(v_n),\partial^y \mathcal{F}(v_n),\partial^z \mathcal{F}(v_n),\partial^w \mathcal{F}(v_n))$, where
\begin{align}\label{eq:subdifferential_of_Lyapunov}
\begin{cases}
\partial^x \mathcal{F}(v_n) &=\nabla f(x_n) +\frac{1}{\gamma}(x_n -y_n) +\frac{2(1 -\nu)}{\gamma}(x_n -z_n),\\
\partial^y \mathcal{F}(v_n) &=-\frac{1}{\gamma}(x_n -y_n) -\frac{1}{\gamma}(y_n -z_n) = -\frac{1}{\gamma}(x_n -z_n),\\
\partial^z \mathcal{F}(v_n) &=\partial h(z_n) -w_n + \frac{1}{\gamma}(y_n -z_n) -\frac{2(1 -\nu)}{\gamma}(x_n -z_n),\\
\partial^w \mathcal{F}(v_n) &=\partial g^*(w_n) -z_n.
\end{cases}
\end{align}
It follows from Lemma~\ref{l:Rela}\ref{l:Rela_y}--\ref{l:Rela_w} and the updating step of $y_{n+1}$ that
\begin{align*}
\begin{cases}
-\frac{1}{\gamma}(x_n -y_{n-1}) &= \nabla f(x_n),\\
z_{n-1} -\tau (w_n -w_{n-1}) &\in \partial g^* (w_n),\\
w_n +\frac{1}{\gamma}(x_n -z_n) +\frac{1}{\gamma}(x_n -y_{n-1}) &\in \partial h(z_n),\\
-\frac{1}{\nu}(y_n -y_{n-1}) &= x_n -z_n.
\end{cases}
\end{align*}
Substitute this to \eqref{eq:subdifferential_of_Lyapunov}, we obtain that
\begin{align*}
\begin{cases}
-\frac{2-\nu}{\nu\gamma}(y_n -y_{n-1}) &=\partial^x \mathcal{F}(v_n), \\
\frac{1}{\nu\gamma}(y_n-y_{n-1}) &= \partial^y \mathcal{F}(v_n), \\
-\frac{1}{\gamma}(y_n -y_{n-1}) &\in \partial^z \mathcal{F}(v_n), \\
-(z_n -z_{n-1}) -\tau (w_n -w_{n-1}) &\in \partial^w \mathcal{F}(v_n).
\end{cases}
\end{align*}
By Lemma~\ref{l:Rela}\ref{l:Rela_yy}\&\ref{l:Rela_zz}, we further derive that
\begin{align*}
\dist(0, \partial^x \mathcal{F}(v_n)) &\leq \frac{(2-\nu)(1 +\gamma\ell)}{\nu\gamma} \|x_{n+1} -x_n\|, \\
\dist(0, \partial^y \mathcal{F}(v_n)) &\leq \frac{1 +\gamma\ell}{\nu\gamma} \|x_{n+1} -x_n\|, \\
\dist(0, \partial^z \mathcal{F}(v_n)) &\leq  \frac{1 +\gamma\ell}{\gamma} \|x_{n+1} -x_n\|,\\
\dist(0, \partial^w \mathcal{F}(v_n)) &\leq \frac{1 +\gamma\ell}{\nu}\|x_{n+1} -x_n\| +\left(1 +\frac{1 +\gamma\ell}{\nu}\right)\|x_n -x_{n-1}\| + \tau \|w_n-w_{n-1}\|.
\end{align*}
Therefore, there exists a constant $C \in \mathbb{R}_{++}$ such that
\begin{align*}
{\color{black}\dist (0,\partial\mathcal{F}(v_n)) } \leq C(\|x_{n+1} -x_n\| +\|x_n -x_{n-1}\| +\tau\|w_n -w_{n-1}\|) \leq C(\Delta_n +\Delta_{n-1}).
\end{align*}

\ref{t:full_cvg}: It can be verified that all the conditions in the \emph{abstract convergence} framework \cite{BDL22} are satisfied with $I =\{0,1\}$, $\lambda_1=\lambda_2=1/2$, $\alpha_n \equiv \frac{\delta}{2 +2\delta\tau}$, {\color{black}$\beta_n \equiv 1/(2C)$}, and $\varepsilon_n \equiv 0 $. By \cite[Theorem~5.1(i)]{BDL22},
\begin{align*}
\sum_{n=0}^{+\infty} (\|x_{n+1} -x_{n}\| +\tau\|w_{n+1}-w_{n}\|) =\sum_{n=0}^{+\infty} \Delta_n < +\infty,
\end{align*}
which implies that $\sum_{n=0}^{+\infty} \|x_{n+1}-x_{n}\| <+\infty$ and $\sum_{n=0}^{+\infty} \|w_{n+1}-w_{n}\| < +\infty$. Combining with Lemma~\ref{l:Rela}\ref{l:Rela_yy}\&\ref{l:Rela_zz}, we derive that 
\begin{align*}
\sum_{n=0}^{+\infty}\|v_n -v^*\| < +\infty,
\end{align*}
which implies the convergence of $(v_n)_{n\in \mathbb{N}}$ to $v^* =(x^*,y^*,z^*,w^*)$. In view of Theorem~\ref{theorem:sub_convergence}\ref{bounded_of_sequence}, $x^* =z^*$ and $0\in \nabla f(z^*) +\partial h(z^*) -\partial g(z^*)$.

\ref{t:full_rate_finite}: This follows from the arguments in the proof of \cite[Theorem~5.1]{BDL22} and \cite[Theorem~2(i)]{Attouch2007}.

\ref{t:full_rate_linear}: By \cite[Theorem~5.1(iv)]{BDL22}, there exist $\Gamma_0\in \mathbb{R}_{++}$ and $\zeta\in \left(0,1\right)$ such that, for all $n\in \mathbb{N}^*$, 
\begin{align*}
\|(x_n, w_n) -(x^*, w^*)\|\leq \Gamma_0\zeta^{\frac{n}{2}}, \text{~~and~~} |\mathcal{F}(v_n) -\mathcal{F}(v^*)|\leq \Gamma_0\zeta^n.  
\end{align*}
It follows that, for all $n\in \mathbb{N}^*$, $\|x_n -x^*\|\leq \Gamma_0\zeta^{\frac{n}{2}}$ and $\|w_n -w^*\|\leq \Gamma_0\zeta^{\frac{n}{2}}$. By passing to the limit in Lemma~\ref{l:Rela}\ref{l:Rela_y}, $y^* =x^* +\gamma\nabla f(x^*)$. Again using Lemma~\ref{l:Rela}\ref{l:Rela_y} and then the Lipschitz continuity of $\nabla f$, we see that, for all $n\in \mathbb{N}^*$,
\begin{align*}
\|y_{n-1} -y^*\|\leq \|x_n -x^*\| +\gamma\|\nabla f(x_n) -\nabla f(x^*))\|\leq (1 +\gamma\ell)\|x_n -x^*\| \leq (1 +\gamma\ell)\Gamma_0\zeta^{\frac{n}{2}}
\end{align*}
and then
\begin{align*}
\|y_n -y_{n-1}\|\leq \|y_n -y^*\| +\|y_{n-1} -y^*\|\leq (1 +\zeta^{\frac{1}{2}})(1 +\gamma\ell)\Gamma_0\zeta^{\frac{n}{2}} < 2(1 +\gamma\ell)\Gamma_0\zeta^{\frac{n}{2}}.    
\end{align*}
Since $x^* =z^*$, the updating step of $y_{n+1}$ implies that, for all $n\in \mathbb{N}^*$, $z_n -z^* =\frac{1}{\nu}(y_n -y_{n-1}) +(x_n - x^*)$, and so
\begin{align*}
\|z_n -z^*\| \leq \|x_n -x^*\| +\frac{1}{\nu}\|y_n -y_{n-1}\| \leq \left(1 +\frac{2(1 +\gamma\ell)}{\nu}\right)\Gamma_0\zeta^{\frac{n}{2}}.
\end{align*}
Consequently,

\begin{align*}
\|v_n -v^*\| \leq \|x_n -x^*\| +\|y_n -y^*\| +\|z_n -z^*\| +\|w_n -w^*\| \leq \left( 3 +2\left(1 +\frac{1}{\nu}\right)(1 +\gamma\ell) \right)\Gamma_0\zeta^{\frac{n}{2}}.
\end{align*}


Next, from \eqref{eq:Fn}, \eqref{eq:Fn'}, and \eqref{eq:wntn bounded} in the proof of Theorem~\ref{theorem:sub_convergence}\ref{bounded_of_sequence}, we see that, for all $n\in \mathbb{N}^*$, 
\begin{align*}
|\mathcal{F}(v_n) -{\color{black}\mathbf{F}(z_n)}|\leq \|t_n -w_n\|(\|z_n -z_{n-1}\| +\tau \|w_n -w_{n-1}\|) +\left(\frac{1}{2\gamma} +\frac{\ell}{2}\right)\|x_n -z_n\|^2
\end{align*}
and that $(t_n)_{n\in\mathbb{N}^*}$ and $(w_n)_{n\in\mathbb{N}^*}$ are bounded. On the other hand, for all $n\geq 2$,
\begin{align*}
\|z_n -z_{n-1}\| &\leq \|z_n -z^*\| +\|z_{n-1} -z^*\|\leq (1 +\zeta^{-\frac{1}{2}})\left(1 +\frac{2(1 +\gamma\ell)}{\nu}\right)\Gamma_0\zeta^{\frac{n}{2}}, \\
\|w_n -w_{n-1}\| &\leq \|w_n -w^*\| +\|w_{n-1} -w^*\|\leq (1 +\zeta^{-\frac{1}{2}})\Gamma_0\zeta^{\frac{n}{2}}, \\
\|x_n -z_n\| &= \frac{1}{\nu}\|y_n -y_{n-1}\|\leq \frac{2(1 +\gamma\ell)}{\nu}\Gamma_0\zeta^{\frac{n}{2}}. 
\end{align*}
Therefore, there exists $\Gamma_1 \in (0,+\infty)$ such that, for all $n\in\mathbb{N}^*$, $|\mathcal{F}(v_n) -{\color{black}\mathbf{F}(z_n)}|\leq \Gamma_1\zeta^{\frac{n}{2}}$. Since $\mathcal{F}(v^*) ={\color{black}\mathbf{F}(z^*)}$, we now have 
\begin{align*}
|{\color{black}\mathbf{F}(z_n) -\mathbf{F}(z^*)}|\leq |{\color{black}\mathbf{F}(z_n)} -\mathcal{F}(v_n)| +|\mathcal{F}(v_n) -\mathcal{F}(v^*)|\leq (\Gamma_1 +\Gamma_0\zeta^{\frac{n}{2}})\zeta^{\frac{n}{2}}\leq (\Gamma_0 +\Gamma_1)\zeta^{\frac{n}{2}}.
\end{align*}
By setting $\Gamma =\max\left\{\left( 3 +2\left(1 +\frac{1}{\nu}\right)(1 +\gamma\ell) \right)\Gamma_0, \Gamma_0 +\Gamma_1\right\}$, we obtain the conclusion.

\ref{t:full_rate_sublinear}: Proceeding as in \cite[Theorem~5.1(iv)]{BDL22} and \cite[Theorem~2(iii)]{Attouch2007}, we find $\Gamma_0 \in \mathbb{R}_{++}$ such that, for all $n\in \mathbb{N}^*$, 
\begin{align*}
\|(x_{n},w_{n}) -(x^*,w^*)\|\leq \Gamma_0 n^{-\frac{1-\theta}{2\theta-1}}\text{~~and~~} |\mathcal{F}(v_n) -\mathcal{F}(v^*)|\leq \Gamma_0 n^{-\frac{2-2\theta}{2\theta-1}}.
\end{align*}
By following the same line of arguments as in \ref{t:full_rate_linear}, we also obtain that, for all $n\in \mathbb{N}^*$,
\begin{align*}
\|v_n -v^*\| \leq \left( 3 +2\left(1 +\frac{1}{\nu}\right)(1 +\gamma\ell) \right)\Gamma_0 n^{-\frac{1-\theta}{2\theta-1}}
\end{align*}
and that, for all $n\geq 2$,
\begin{align*}
\|z_n -z_{n-1}\| &\leq \left( 1 +\frac{2(1+\gamma\ell)}{\nu} \right)\Gamma_0 \left( n^{-\frac{1-\theta}{2\theta-1}} +(n-1)^{-\frac{1-\theta}{2\theta-1}} \right), \\
\|w_n -w_{n-1}\| &\leq \Gamma_0\left( n^{-\frac{1-\theta}{2\theta-1}} +(n-1)^{-\frac{1-\theta}{2\theta-1}} \right), \\
\|x_n -z_n\| & \leq \frac{2(1+\gamma\ell)}{\nu}\Gamma_0 n^{-\frac{1-\theta}{2\theta-1}}. 
\end{align*}
We observe that $n-1 \geq \frac{1}{2}n$ for $n\geq 2$ and that $\frac{1-\theta}{2\theta-1} >0$ since $\theta \in (\frac{1}{2},1)$. Thus, for all $n\geq 2$, $(n-1)^{-\frac{1-\theta}{2\theta-1}}\leq 2^{\frac{1-\theta}{2\theta-1}}n^{-\frac{1-\theta}{2\theta-1}}$. As in \ref{t:full_rate_linear}, there exists $\Gamma_1 \in (0,+\infty)$ such that, for all $n\in\mathbb{N}^*$, $|\mathcal{F}(v_n) -{\color{black}\mathbf{F}(z_n)}|\leq \Gamma_1 n^{-\frac{1-\theta}{2\theta-1}}$, which leads to 
\begin{align*}
|{\color{black}\mathbf{F}(z_n) -\mathbf{F}(z^*)}|\leq (\Gamma_1 +\Gamma_0 n^{-\frac{1-\theta}{2\theta-1}}) n^{-\frac{1-\theta}{2\theta-1}}\leq (\Gamma_0 +\Gamma_1)n^{-\frac{1-\theta}{2\theta-1}},
\end{align*}
Setting $\Gamma =\max\left\{\left( 3 +2\left(1 +\frac{1}{\nu}\right)(1 +\gamma\ell) \right)\Gamma_0, \Gamma_0 +\Gamma_1\right\}$, we complete the proof.
\end{proof}

To conclude this section, we observe that Theorem~\ref{t:full_sequence_convergence} requires $\mathcal{F}$ to be a KL function to ensure the convergence of the full sequence generated by Algorithm~\ref{algo:CDR}. Notably, if the objective function $F$ is semi-algebraic, then $\mathcal{F}$ is also semi-algebraic and hence a KL function with exponent $\theta \in [0, 1)$; see, e.g., \cite[Example~1]{Attouch2007}.

\section{Numerical results}
\label{sec:numerical_result}

All of the experiments are performed in MATLAB R2023b on a 64-bit laptop with Intel(R) Core(TM) i7-1165G7 CPU (2.80GHz) and 32GB of RAM. 

We consider a compressed sensing problem to evaluate the performance of the proposed BDR. Practical applications of compressed sensing are presented in \cite{CSAPPS,Babakmehr2018}. One important application of compressed sensing is recovering a sparse signal from a set of measurements. A basic mathematical model for this application is the \emph{basic pursuit problem}, which is also known as the \emph{Lasso problem} or $\ell_1-$\emph{regularized problem}, as below \cite{Tibshirani1996}:
\begin{align}
    \min_{x\in \mathbb{R}^{d}}~~\frac{1}{2}\|Ax-b\|^2 + \lambda \|x\|_1, \tag{CS-1}
\end{align}
where $\lambda\in \mathbb{R}_{++}$ is a regularization parameter, $A\in \mathbb{R}^{m\times d}$ is a sensing matrix, and $b \in \mathbb{R}^{d} \setminus \{0\}$. The research in \cite{Lou2017} has shown that the $\ell_{1-2}$, which is the difference of $\ell_1$ and $\ell_2$, has better performance. The model in \cite{Lou2017} is as
\begin{align}\label{ex:CSM}
    \min_{x\in \mathbb{R}^{d}}~~\frac{1}{2}\|Ax-b\|^2 + \lambda (\|x\|_1-\|x\|). \tag{CS-2}
\end{align}
Now problem \eqref{ex:CSM} takes the form of \eqref{eq:P} with $f(x)=\frac{1}{2}\|Ax-b\|^2$, $h(x)=\lambda \|x\|_1$, and $g(x)=\lambda \|x\|$. {\color{black}Note that the objective function in (CS-2) is semialgebraic, and hence satisfies the KL property}. The updating steps of the BDR for solving \eqref{ex:CSM} are as
\begin{align*}
    \begin{cases}
        x_{n+1}&=(A^\top A + \frac{1}{\gamma}I)^{-1}(A^\top b +\frac{y_n}{\gamma}),\\
        w_{n+1}&=\min \left\{ \frac{\lambda}{\|\tau w_n +z_n\|},\frac{1}{\tau}\right\}(\tau w_n +z_n),\\
        z_{n+1}&=\mathcal{S}_{\gamma \lambda}(2x_{n+1} - z_n + \gamma w_{n+1})\\
        y_{n+1} &=y_n +\nu(z_{n+1} -x_{n+1}),
    \end{cases}
\end{align*}
where the updating step of $w_{n+1}$ is derived from \eqref{eq:Moreau} and the proximity operator of the $\ell_2$ norm given in \cite[Example~24.20]{Bauschke2017_book} as
\begin{align*}
\prox_{\tau g}(\tau w_n +z_n)=\prox_{\tau\lambda \|\cdot\|}(\tau w_n +z_n) =\begin{cases}
\tau w_n +z_n-\frac{\tau \lambda (\tau w_n +z_n)}{\|\tau w_n +z_n\|} \quad &\text{if~} \|\tau w_n +z_n\|>\tau \lambda, \\
0 \quad &\text{if~} \|\tau w_n +z_n\| \leq \tau \lambda,
\end{cases}
\end{align*} 
while $\mathcal{S}_{\gamma \lambda}(\cdot)$ is the \emph{soft shrinkage} operator \cite[Example 24.22]{Bauschke2017_book} given by, for all $i=1,\dots,d$,  
\begin{align*}
\color{black}
(\mathcal{S}_{\gamma \lambda}(u))_i =\operatorname{sign}(u_i) \max\{0, \lvert u_i\rvert-\gamma\lambda\}.
\end{align*}%

  

In Section~\ref{subsec:sparse_recon_random_load}, we evaluate the performance of the BDR and compare it with notable algorithms in the literature, in terms of sparse signal reconstruction, on both randomly generated data and a public dataset. In Section~\ref{subsec:profile_reconstruction}, we compare the performance of the BDR with the benchmark methods in terms of reconstructing signals from incomplete signals, taken into account their sparse representation.

\subsection{Sparse signal reconstruction on synthetic data and real power system voltage data}\label{subsec:sparse_recon_random_load}
In this numerical experiment, we randomly generate the matrix $A$ and the ground truth sparse vector $x_g$ based on the method given in \cite{Lou2017}. Let $b=Ax_g + 10^{-3} * z$, where $z \in \mathbb{R}^d$ is a vector with i.i.d. standard Gaussian entries. We consider randomly generated Gaussian matrices and discrete cosine transform (DCT) matrices. For each case, we consider different matrix sizes of $m \times d$ with sparsity level $s$ as given in Table~\ref{tab:caseCS}. 
\begin{table}[H]
\caption{Test cases}
\label{tab:caseCS}
\begin{center}
\begin{tabular}{c|c|ccc|c|c|ccc}
\hline
Sensing matrix type        & Case & $m$  & $d$   & $s$ & Sensing matrix type   & Case & $m$  & $d$   & $s$ \\ \hline
\multirow{10}{*}{Gaussian} & 1    & 360  & 1280  & 40  & \multirow{10}{*}{DCT} & 11   & 360  & 1280  & 40  \\
                           & 2    & 720  & 2560  & 80  &                       & 12   & 720  & 2560  & 80  \\
                           & 3    & 1080 & 3840  & 120 &                       & 13   & 1080 & 3840  & 120 \\
                           & 4    & 1440 & 5120  & 160 &                       & 14   & 1440 & 5120  & 160 \\
                           & 5    & 1800 & 6400  & 200 &                       & 15   & 1800 & 6400  & 200 \\
                           & 6    & 2160 & 7680  & 240 &                       & 16   & 2160 & 7680  & 240 \\
                           & 7    & 2520 & 8960  & 280 &                       & 17   & 2520 & 8960  & 280 \\
                           & 8    & 2880 & 10240 & 320 &                       & 18   & 2880 & 10240 & 320 \\
                           & 9    & 3240 & 11520 & 360 &                       & 19   & 3240 & 11520 & 360 \\
                           & 10   & 3600 & 12800 & 400 &                       & 20   & 3600 & 12800 & 400 \\ \hline
\end{tabular}
\end{center}
\end{table}

 For the BDR, we set $\tau =20$, and $\rho=0$ (as $f$ is convex). The parameter $\gamma$ is set to
\begin{align*}
\gamma = \frac{-\nu\rho +\sqrt{\nu^2\rho^2 -8(\nu -2)\ell^2}}{4\ell^2} -10^{-10},
\end{align*}
 where $\ell$ is maximum eigenvalue of the symmetric matrix $A^\top A$. We compare our algorithm with the following algorithms
\begin{itemize}
    \item Alternating direction method of multipliers (ADMM) with fast $\ell_{1-2}$ proximal operator in \cite{Lou2017};
    \item Unified Douglas--Rachford algorithm for difference of convex (DR-DC) in \cite{Chuang2021};
    \item Proximal difference-of-convex algorithm with extrapolation (pDCAe) in \cite{Wen2017};
    \item Hybrid Bregman alternating direction method of multipliers (HBADMM) in \cite{Tu2019}.
\end{itemize}
The parameters of these algorithms are set to the same ones given in their corresponding references. We set $\lambda=0.1$ and run all algorithms, initialized at the origin, for 30 times with a maximum of 3000 iterations. At each time we randomly regenerate $A$ and $x_g$. The stopping condition for all algorithms is $\frac{\|z_{n+1}-z_n\|}{\|z_n\|}<10^{-6}$. To finetune and choose $\nu$, we take Case 1 in Table~\ref{tab:caseCS} for an empirical study. We run the BDR on Case 1 with different values of $\nu$ and observe its convergence. Figure~\ref{fig:fine_tune_nu} shows the average results after 30 runs of this empirical study, hence we choose $\nu=1.4$. In Tables~\ref{tab: CPU_time_iter}\&\ref{tab:errorvsground}, we report the CPU time, the number of iterations, the error to the ground truth at termination (defined as $\frac{\|z_{n+1}-x_g\|}{\|x_g\|}$), averaged over 30 random instances. It is seen that in terms of solution quality, the BDR has the same error values of the other algorithms. However, the average running time of the BDR is shorter than the average running time of all other algorithms in these tables. 

\begin{figure}[H]
\centering
\includegraphics[width=0.8\linewidth]{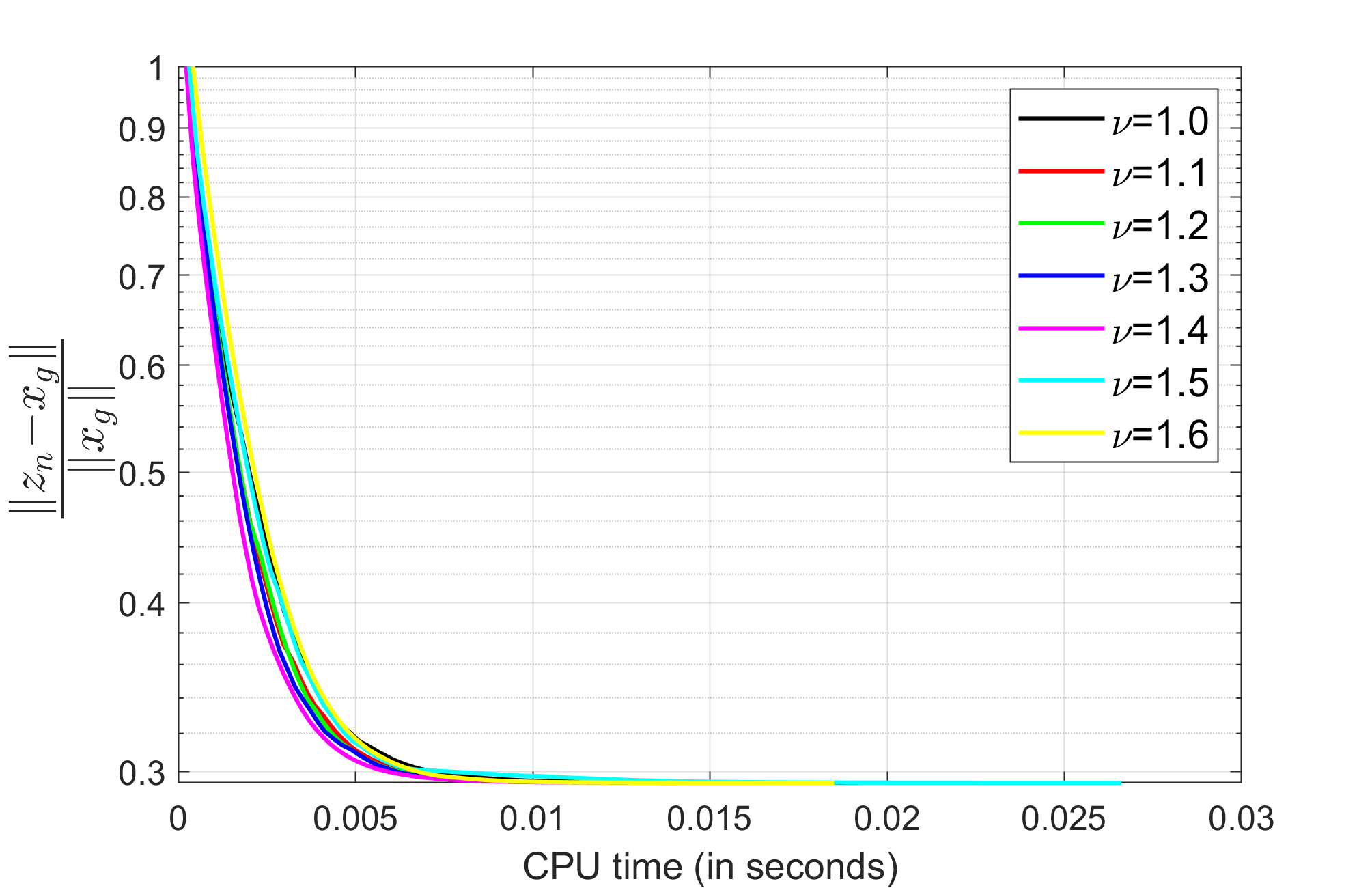}
\caption{Convergence of the BDR with different $\nu$ values.}
\label{fig:fine_tune_nu}
\end{figure}

\begin{table}[H]
\centering
\caption{CPU time and number of iterations required for each algorithm, averaged over 30 random generated instances}
\label{tab: CPU_time_iter}
\setlength{\tabcolsep}{1.5pt}
\begin{tabular}{c|ccccc|ccccc}
\hline
\multicolumn{1}{l|}{}     & \multicolumn{5}{c|}{CPU time (in seconds)}                                                                                                                     & \multicolumn{5}{c}{Number of iterations}                                                                                                                    \\ \hline
\multicolumn{1}{l|}{Case} & \multicolumn{1}{l}{ADMM} & \multicolumn{1}{l}{DR-DC} & \multicolumn{1}{l}{pDCAe} & \multicolumn{1}{l}{HBADMM} & \multicolumn{1}{l|}{\textbf{BDR}} & \multicolumn{1}{l}{ADMM} & \multicolumn{1}{l}{DR-DC} & \multicolumn{1}{l}{pDCAe} & \multicolumn{1}{l}{HBADMM} & \multicolumn{1}{l}{\textbf{BDR}} \\ \hline
1                         & 0.25                     & 0.15                      & 0.05                      & 0.11                       & \textbf{0.04}                              & 1206                     & 727                       & 128                       & \textbf{96}                        & 144                             \\
2                         & 1.53                     & 0.89                      & 0.37                      & 0.46                       & \textbf{0.26}                             & 1143                     & 687                       & 134                       &\textbf{80}                         & 198                              \\
3                         & 4.01                     & 2.41                      & 0.89                      & 1.30                       & \textbf{0.69}                              & 1080                     & 649                       & 132                       & \textbf{89}                         & 186                              \\
4                         & 7.69                     & 4.47                      & 1.79                      & 2.62                       & \textbf{1.36}                              & 1096                     & 657                       & 138                       & \textbf{101}                        & 187                              \\
5                         & 10.44                    & 6.25                      & 2.36                      & 2.74                       & \textbf{1.76}                              & 1057                     & 634                       & 135                       & \textbf{75}                         & 181                              \\
6                         & 14.71                    & 8.77                      & 3.52                      & 5.32                       & \textbf{2.50}                              & 1069                     & 640                       & 139                       & \textbf{103}                       & 182                              \\
7                         & 19.74                    & 11.81                     & 4.58                      & 5.2                        & \textbf{3.36}                              & 1047                     & 627                       & 138                       & \textbf{75}                         & 178                              \\
8                         & 25.16                    & 15.13                        & 6.14                      & 6.59                       & \textbf{4.27}                              & 1039                     & 622                       & 138                       & \textbf{72}                         & 177                              \\
9                         & 32.78                    & 19.56                     & 7.84                       & 9.82                        & \textbf{5.56}                              & 1042                     & 624                       & 139                       & \textbf{85}                         & 177                              \\
10                        & 40.25                    & 23.98                     & 9.78                      & 11.36                      & \textbf{6.81}                              & 1038                     & 621                       & 138                       & \textbf{80}                        & 176                        \\
\hline   
11                        & 0.11                     & 0.07                      & 0.03                      & 0.09                       & \textbf{0.02}                              & 545                      & 323                       & \textbf{72}                        & 76                         & 90                               \\
12                        & 0.63                     & 0.37                      & 0.19                      & 0.43                       & \textbf{0.11}                              & 531                      & 315                       & \textbf{72}                        & 80                         & 87                               \\
13                        & 1.68                     & 0.97                      & 0.45                      & 1.26                       & \textbf{0.27}                              & 522                      & 309                       & \textbf{73}                        & 99                         & 86                               \\
14                        & 3.18                     & 1.85                      & 0.82                      & 2.47                       & \textbf{0.52}                              & 526                      & 312                       & \textbf{74}                        & 107                        & 86                               \\
15                        & 4.99                     & 2.94                      & 1.24                      & 4.23                       & \textbf{0.81}                              & 523                      & 310                       & \textbf{74}                        & 118                        & 86                               \\
16                        & 7.34                     & 4.35                      & 1.96                       & 6.42                       & \textbf{1.23}                               & 520                      & 308                       & \textbf{73}                        & 120                        & 85                               \\
17                        & 10.42                    & 6.05                      & 2.71                      & 9.72                       & \textbf{1.68}                              & 518                      & 307                       & \textbf{74}                        & 133                        & 85                               \\
18                        & 13.48                    & 7.88                      & 3.55                      & 12.41                      & \textbf{2.25}                              & 520                      & 308                       & \textbf{74}                        & 130                        & 85                               \\
19                        & 16.58                    & 9.93                       & 4.23                      & 16.07                      & \textbf{2.76}                              & 518                      & 307                       & \textbf{74}                        & 135                        & 85                               \\
20                        & 20.25                    & 11.96                     & 5.24                      & 22.37                      & \textbf{3.31}                              & 517                      & 306                       & \textbf{74}                        & 156                        & 84                               \\ \hline
\end{tabular}
\end{table}

\begin{table}[H]
\centering
\caption{Error vs ground truth, averaged over 30 random instances}
\label{tab:errorvsground}
\begin{tabular}{c|ccccc}
\hline
Case & ADMM     & DR-DC    & pDCAe    & HBADMM   & \textbf{BDR} \\ \hline
1    & 3.08E-01 & 3.08E-01 & 3.08E-01 & 3.08E-01 & 3.08E-01    \\
2    & 6.08E-01 & 6.08E-01 & 6.08E-01 & 6.08E-01 & 6.08E-01     \\
3    & 6.17E-01 & 6.17E-01 & 6.17E-01 & 6.17E-01 & 6.17E-01     \\
4    & 6.28E-01 & 6.28E-01 & 6.28E-01 & 6.28E-01 & 6.28E-01     \\
5    & 6.24E-01 & 6.24E-01 & 6.24E-01 & 6.24E-01 & 6.24E-01     \\
6    & 6.29E-01 & 6.29E-01 & 6.29E-01 & 6.29E-01 & 6.29E-01     \\
7    & 6.20E-01 & 6.20E-01 & 6.20E-01 & 6.20E-01 & 6.20E-01     \\
8    & 6.25E-01 & 6.25E-01 & 6.25E-01 & 6.25E-01 & 6.25E-01     \\
9    & 6.22E-01 & 6.22E-01 & 6.22E-01 & 6.22E-01 & 6.22E-01     \\
10   & 6.29E-01 & 6.29E-01 & 6.29E-01 & 6.29E-01 & 6.29E-01     \\ \hline
11   & 3.10E-01 & 3.10E-01 & 3.10E-01 & 3.10E-01 & 3.10E-01     \\
12   & 3.22E-01 & 3.22E-01 & 3.22E-01 & 3.22E-01 & 3.22E-01     \\
13   & 3.29E-01 & 3.29E-01 & 3.29E-01 & 3.29E-01 & 3.29E-01     \\
14   & 3.25E-01 & 3.25E-01 & 3.25E-01 & 3.25E-01 & 3.25E-01     \\
15   & 3.28E-01 & 3.28E-01 & 3.28E-01 & 3.28E-01 & 3.28E-01     \\
16   & 3.28E-01 & 3.28E-01 & 3.28E-01 & 3.28E-01 & 3.28E-01     \\
17   & 3.31E-01 & 3.31E-01 & 3.31E-01 & 3.31E-01 & 3.31E-01     \\
18   & 3.31E-01 & 3.31E-01 & 3.31E-01 & 3.31E-01 & 3.31E-01     \\
19   & 3.31E-01 & 3.31E-01 & 3.31E-01 & 3.31E-01 & 3.31E-01     \\
20   & 3.32E-01 & 3.31E-01 & 3.31E-01 & 3.31E-01 & 3.31E-01     \\ \hline
\end{tabular}
\end{table}

Next, we evaluate the performance of the BDR on a public dataset. We extract the sinusoidal voltage signal from the ``Electrical Signals Databases''\footnote{This dataset can be found at \url{https://github.com/rte-france/digital-fault-recording-database}.}. This signal is sparse when the \textit{discrete cosine transform} (DCT) is applied to it, as shown in Figure~\ref{fig:voltage_sparsity}. We take 12800 samples (corresponding to the highest $d$ value in Table~\ref{tab:caseCS}), and then compare the performance of the BDR and other algorithms on this signal, based on the same experimental setup described above. In this case, each time we take a signal of length $d$ and apply the DCT to it to obtain $x_g$. All of the other parameters are set to the same values as in the previous numerical experiment. From Tables~\ref{tab: CPU_time_iter_vol}\&\ref{tab:errorvsground_voltage}, it is seen that the BDR converges faster than all other algorithms in all cases while maintaining the same solution quality. This comparison shows the higher computational efficiency of BDR, which is a key issue in practice for the operation of electric power systems.
\begin{figure}[H]
\centering
\includegraphics[width=0.45\linewidth]{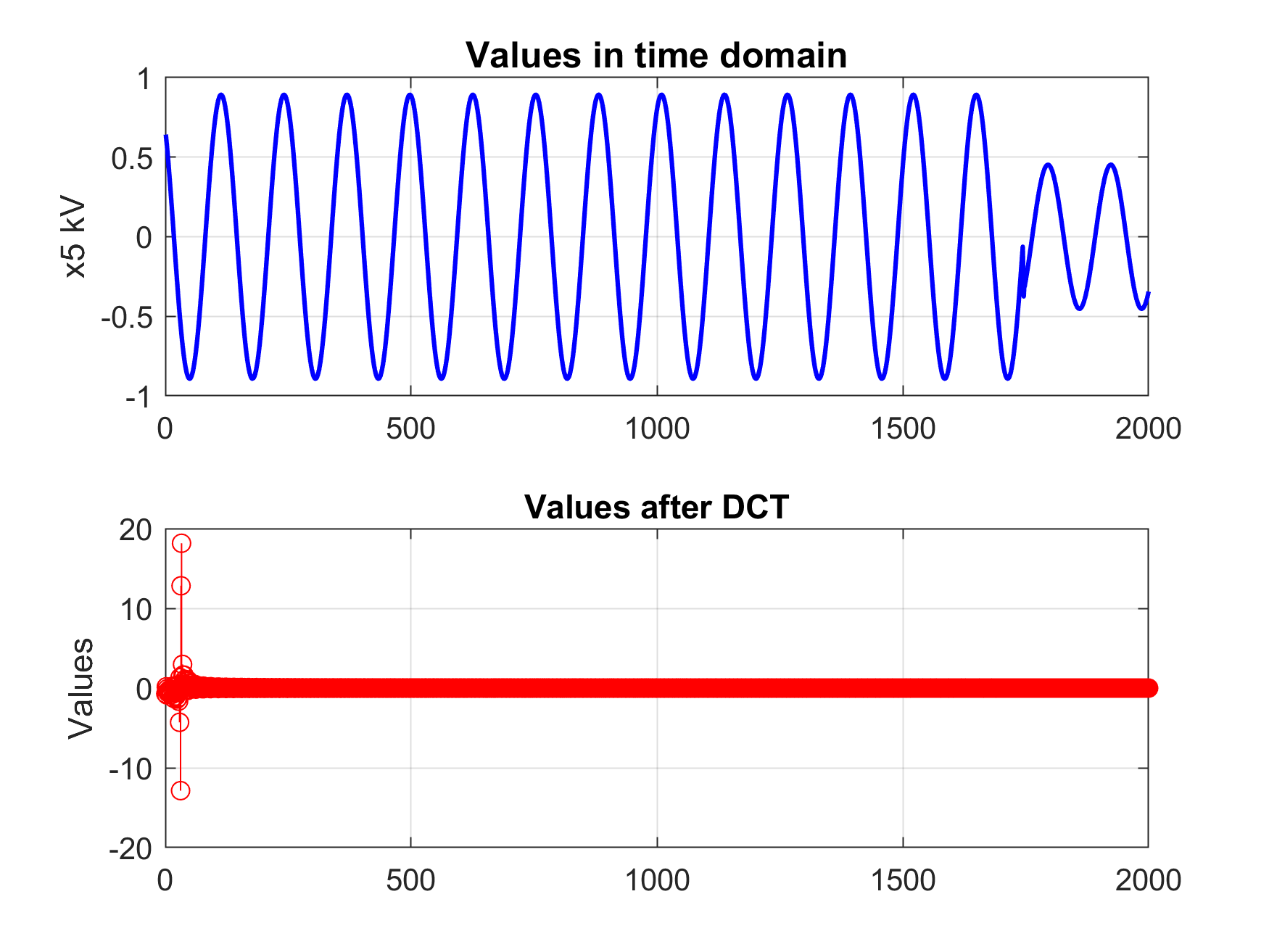}
\caption{Electric voltage signal and its sparsity (for illustration purpose, only the first 2000 samples of the signal are shown here).}
\label{fig:voltage_sparsity}
\end{figure}

\begin{table}[H]
\centering
\caption{CPU time and number of iterations required for each algorithm, averaged over 30 runs on electric voltage signal dataset}
\label{tab: CPU_time_iter_vol}
\setlength{\tabcolsep}{1pt}
\begin{tabular}{cccccc|ccccc}
\hline
\multicolumn{1}{l}{} & \multicolumn{5}{c|}{CPU time (in seconds)}    & \multicolumn{5}{c}{Number of iterations}     \\ \hline
Case                 & ADMM  & DR-DC & pDCAe & HBADMM & \textbf{BDR} & ADMM & DR-DC & pDCAe & HBADMM & \textbf{BDR} \\ \hline
1                    & 0.12  & 0.07  & 0.03  & 0.07   & \textbf{0.02}         & 594  & 347   & 78    & \textbf{59}     & 92           \\
2                    & 0.83  & 0.42  & 0.23  & 0.40   & \textbf{0.14}         & 560  & 328   & 76    & \textbf{69}     & 87           \\
3                    & 1.93  & 1.08  & 0.48  & 0.94   & \textbf{0.29}         & 533  & 312   & 74    & \textbf{72}     & 84           \\
4                    & 3.30  & 1.90  & 0.86  & 1.74   & \textbf{0.51}         & 517  & 303   & \textbf{73}    & 75     & 81           \\
5                    & 5.00  & 2.90  & 1.24  & 3.72   & \textbf{0.77}         & 497  & 292   & \textbf{71}    & 103    & 78           \\
6                    & 7.02  & 4.09  & 1.78  & 3.93   & \textbf{1.10}         & 489  & 286   & \textbf{69}  & 77     & 77           \\
7                    & 9.37  & 5.61  & 2.58  & 6.97   & \textbf{1.62}         & 477  & 280   & \textbf{68}    & 91     & 75           \\
8                    & 13.14 & 7.70  & 3.64  & 4.72   & \textbf{1.96}         & 474  & 278   & 68    & \textbf{50}     & 75           \\
9                    & 15.48 & 9.05  & 3.89  & 6.45   & \textbf{2.43}         & 464  & 272   & 67    & \textbf{56}    & 73           \\
10                   & 21.18 & 12.86 & 5.82  & 8.99   & \textbf{3.48}         & 456  & 267   & 64    & \textbf{57}     & 72           \\ \hline
11                   & 0.12  & 0.07  & 0.03  & 0.07   & \textbf{0.02}         & 582  & 340   & 74    & \textbf{62}     & 90           \\
12                   & 0.76  & 0.42  & 0.22  & 0.38   & \textbf{0.12}         & 580  & 340   & 78    & \textbf{69}     & 91           \\
13                   & 1.82  & 1.04  & 0.47  & 0.80   & \textbf{0.29}         & 521  & 305   & 71    &\textbf{} \textbf{62}     & 82           \\
14                   & 3.19  & 1.85  & 0.84  & 2.14   & \textbf{0.51}         & 503  & 295   & \textbf{71}    & 92     & 80           \\
15                   & 5.11  & 3.00  & 2.05  & 2.45   & \textbf{1.35}         & 509  & 299   & 116   & \textbf{68}     & 135          \\
16                   & 7.05  & 4.15  & 1.83  & 4.28   & \textbf{1.11}         & 492  & 288   & \textbf{70}    & 84     & 77           \\
17                   & 10.03 & 5.87  & 2.38  & 4.76   & \textbf{1.58}         & 481  & 283   & 67    & \textbf{64}     & 76           \\
18                   & 12.21 & 7.22  & 3.28  & 5.20   & \textbf{1.91}         & 469  & 276   & 67    & \textbf{56}     & 74           \\
19                   & 16.80 & 9.66  & 4.19  & 7.15   & \textbf{2.64}         & 481  & 283   & 69    & \textbf{60}     & 76           \\
20                   & 19.29 & 11.34 & 4.94  & 9.84   & \textbf{3.00}         & 465  & 272   & \textbf{66}    & 68     & 73           \\ \hline
\end{tabular}
\end{table}

\begin{table}[H]
\centering
\caption{Error vs ground truth, averaged over 30 runs on electric voltage signal dataset}
\label{tab:errorvsground_voltage}
\begin{tabular}{c|rrrrr}
\hline
Case & \multicolumn{1}{c}{ADMM} & \multicolumn{1}{c}{DR-DC} & \multicolumn{1}{c}{pDCAe} & \multicolumn{1}{c}{HBADMM} & \multicolumn{1}{c}{\textbf{BDR}} \\ \hline
1    & 7.60E-02                 & 7.60E-02                  & 7.60E-02                  & 7.60E-02                   & 7.60E-02                         \\
2    & 9.69E-02                 & 9.69E-02                  & 9.69E-02                  & 9.69E-02                   & 9.69E-02                         \\
3    & 1.03E-01                 & 1.03E-01                  & 1.03E-01                  & 1.03E-01                   & 1.03E-01                         \\
4    & 1.11E-01                 & 1.11E-01                  & 1.11E-01                  & 1.11E-01                   & 1.11E-01                         \\
5    & 1.11E-01                 & 1.11E-01                  & 1.11E-01                  & 1.11E-01                   & 1.11E-01                         \\
6    & 1.10E-01                 & 1.10E-01                  & 1.10E-01                  & 1.10E-01                   & 1.10E-01                         \\
7    & 1.11E-01                 & 1.11E-01                  & 1.11E-01                  & 1.11E-01                   & 1.11E-01                         \\
8    & 1.12E-01                 & 1.12E-01                  & 1.12E-01                  & 1.12E-01                   & 1.12E-01                         \\
9    & 1.10E-01                 & 1.10E-01                  & 1.10E-01                  & 1.10E-01                   & 1.10E-01                         \\
10   & 1.09E-01                 & 1.09E-01                  & 1.09E-01                  & 1.09E-01                   & 1.09E-01                         \\ \hline
11   & 8.05E-02                 & 8.04E-02                  & 8.04E-02                  & 8.04E-02                   & 8.04E-02                         \\
12   & 9.36E-02                 & 9.36E-02                  & 9.36E-02                  & 9.36E-02                   & 9.36E-02                         \\
13   & 1.05E-01                 & 1.04E-01                  & 1.04E-01                  & 1.04E-01                   & 1.04E-01                         \\
14   & 1.18E-01                 & 1.18E-01                  & 1.18E-01                  & 1.18E-01                   & 1.18E-01                         \\
15   & 1.09E-01                 & 1.09E-01                  & 1.09E-01                  & 1.09E-01                   & 1.09E-01                         \\
16   & 1.13E-01                 & 1.13E-01                  & 1.13E-01                  & 1.13E-01                   & 1.13E-01                         \\
17   & 1.12E-01                 & 1.12E-01                  & 1.12E-01                  & 1.12E-01                   & 1.12E-01                         \\
18   & 1.12E-01                 & 1.12E-01                  & 1.12E-01                  & 1.12E-01                   & 1.12E-01                         \\
19   & 1.11E-01                 & 1.11E-01                  & 1.11E-01                  & 1.11E-01                   & 1.11E-01                         \\
20   & 1.08E-01                 & 1.08E-01                  & 1.08E-01                  & 1.08E-01                   & 1.08E-01                         \\ \hline
\end{tabular}
\end{table}

\subsection{Reconstruction of signals from incomplete and noisy signals}\label{subsec:profile_reconstruction}
In this experiment, we aim to reconstruct a signal from a random set of its noisy measurements, taken into account the fact that the signal is sparse in a specific domain. The model \eqref{ex:CSM} can be used for this task. For example, suppose that we have a signal $\mathbf{u}=[u_1,u_2,u_3,u_4,u_5]^\top$, which is sparse after applying DCT, and suppose that $u_2$, $u_3$, and $u_5$ are missing (in other words, we only know the values of $u_1$ and $u_4$). We can utilize the model \eqref{ex:CSM} to recover the signal $\mathbf{u}$ with
\begin{align*}
    A=S\Psi,~\text{where~} S=
    \begin{bmatrix}
1 & 0 & 0 & 0 & 0 \\
0 & 0 & 0 & 1 & 0 
\end{bmatrix},
\end{align*}
$\Psi$ is the inverse DCT matrix, $b=[u_1, u_4]^\top$, and $x \in \mathbb{R}^5$ is the sparse vector that we want to find. After $x$ is found, the reconstructed signal is obtained by $\Psi x$. In practice, the length of the signal $\mathbf{u}$ is known, and the indices of its missing entries are also known, hence $S$ can be easily constructed.

Using this idea, we first test our algorithm on the electric voltage signal data from Section~\ref{subsec:sparse_recon_random_load}, which is a generally smooth signal. A signal of length $d$ is sampled and Gaussian noise, which typically represents the measurement error in a practical power system \cite{Nejati2012}, is added to it to form a noisy signal. Then $\mathcal{R}$ percent of the elements of the noisy signal is sampled to form the vector $b$. The matrix $S$ is constructed based on the indices of those sampled entries, and $A=S\Psi$ with $\Psi$ being the inverse DCT matrix. We then solve the problem \eqref{ex:CSM} to obtain the reconstructed signal $\mathbf{\widehat{u}}$. Figure~\ref{fig:plot_recon} illustrates an example of this process (here, $40\%$ of the entries are sampled).

\begin{figure}[H]
\centering
\includegraphics[width=0.6\linewidth]{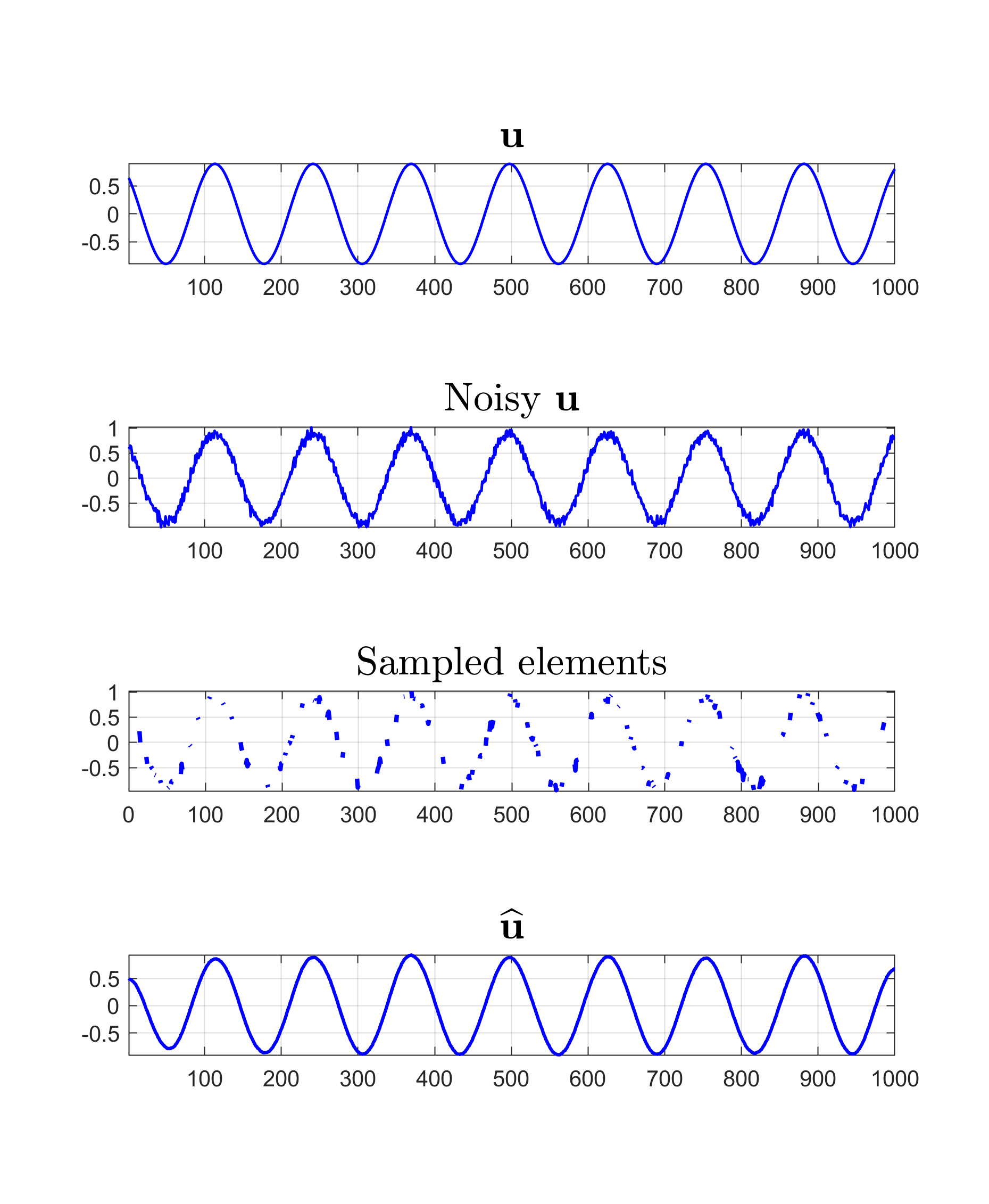}
\caption{An illustration of the case study.}
\label{fig:plot_recon}
\end{figure}
We employ the heuristic in \cite[Remark 4]{Li2015DR} to set the parameter $\gamma$. Firstly, we set $\gamma = k \gamma_0$ and update $\gamma=\max \{\gamma/2,0.9999\gamma_0\}$ whenever $\gamma>\gamma_0$ and the sequence $x_n$ satisfies $\|x_{n+1}-x_n\|>1000/n$ or $\|x_n\|>10^{10}$. We also choose $\eta=1.4$. Then in the view of Theorem~\ref{theorem:sub_convergence}, $\overline{\gamma}\approx 0.547$, and thus, we set $\gamma_0 =0.447$ for our algorithm and let $k=10$. We compare the BDR with the other algorithms for 30 runs, with different sampling rate $\mathcal{R}$ and signal length values. In each run, $\mathcal{R}$ percent of the entries of the signal are randomly sampled. All algorithms start at the origin, the maximum number of iterations is 3000, $\lambda=0.1$, and the stopping criteria is the same as in Section~\ref{subsec:sparse_recon_random_load}. For evaluating the solution quality, we use the following criteria
\begin{align*}
    \text{Error vs ground truth~}\frac{\|\mathbf{u}-\mathbf{\widehat{u}}\|}{\|\mathbf{u}\|}, \text{~and Signal-to-noise ratio (SNR)}:=20\log_{10}\frac{\|\mathbf{u}\|}{\|\mathbf{u}-\mathbf{\widehat{u}}\|}  \cite{Tu2019}.
\end{align*}
The first criterion measures the closeness of the solution to the ground truth while the second one evaluates the quality of the denoised signal. A higher SNR value means a higher denoising quality. Tables~\ref{tab:results_recon_voltage_signal_time_iter}\&\ref{tab:results_recon_voltage_signal_err} show that the BDR converges significantly faster than other algorithms. In terms of solution quality, the error vs ground truth of the BDR is on par with the ones obtained by the other algorithms, while its SNR is even slightly higher than the SNR obtained by the ADMM and DR-DC. The performance of the BDR is consistent on different signal lengths and sampling ratios.

\begin{table}[H]
\centering
\caption{Average CPU time and number of iterations over 30 runs, electric voltage signal reconstruction}
\label{tab:results_recon_voltage_signal_time_iter}
\setlength{\tabcolsep}{1pt}
\begin{tabular}{c|c|ccccc|ccccc}
\hline
\multicolumn{1}{l|}{}  & \multicolumn{1}{l|}{} & \multicolumn{5}{c|}{CPU time (in seconds)}                     & \multicolumn{5}{c}{Number of iterations}               \\ \hline
Length          & $\mathcal{R}$         & ADMM  & DR-DC & pDCAe & HBADMM & \textbf{BDR} & ADMM & DR-DC & pDCAe & HBADMM & \textbf{BDR} \\ \hline
\multirow{3}{*}{2000}  & 20\%                  & 0.32  & 0.19  & 0.17  & 0.21   & \textbf{0.02}         & 956  & 562   & 121   & 93     & \textbf{60}           \\
                       & 30\%                  & 0.67  & 0.38  & 0.10  & 0.29   & \textbf{0.07}         & 413  & 242   & 55    & 52     & \textbf{41}           \\
                       & 40\%                  & 0.72  & 0.41  & 0.10  & 0.30   & \textbf{0.08}         & 366  & 215   & 50    & 49     & \textbf{39}           \\ \hline
\multirow{3}{*}{5000}  & 20\%                  & 2.99  & 1.69  & 1.03  & 2.12   & \textbf{0.22}         & 791  & 467   & 109   & 132    & \textbf{58}           \\
                       & 30\%                  & 3.49  & 2.03  & 0.83  & 2.48   & \textbf{0.32}         & 530  & 312   & 74    & 107    & \textbf{49}           \\
                       & 40\%                  & 4.10  & 2.36  & 0.72  & 2.17   & \textbf{0.44}         & 410  & 241   & 56    & 70     & \textbf{43}           \\ \hline
\multirow{3}{*}{10000} & 20\%                  & 11.25 & 6.59  & 3.78  & 7.73   & \textbf{0.89}         & 731  & 432   & 103   & 125    & \textbf{57}           \\
                       & 30\%                  & 13.11 & 8.21  & 3.27  & 9.47   & \textbf{1.48}         & 503  & 296   & 70    & 102    & \textbf{49}           \\
                       & 40\%                  & 15.24 & 8.93  & 2.61  & 11.40  & \textbf{1.66}         & 387  & 227   & 54    & 94     & \textbf{42}           \\ \hline
\end{tabular}
\end{table}

\begin{table}[H]
\centering
\caption{Average error vs ground truth and SNR over 30 runs, electric voltage signal reconstruction}
\label{tab:results_recon_voltage_signal_err}
\setlength{\tabcolsep}{1pt}
\begin{tabular}{c|c|ccccc|ccccc}
\hline
\multicolumn{1}{l|}{}  & \multicolumn{1}{l|}{} & \multicolumn{5}{c|}{Error vs ground truth}               & \multicolumn{5}{c}{SNR}                          \\ \hline
Length          & $\mathcal{R}$         & ADMM     & DR-DC    & pDCAe    & HBADMM   & \textbf{BDR} & ADMM   & DR-DC  & pDCAe  & HBADMM & \textbf{BDR} \\ \hline
\multirow{3}{*}{2000}  & 20\%                  & 1.27E-01 & 1.27E-01 & 1.27E-01 & 1.27E-01 & 1.27E-01     & 17.957 & 17.958 & 17.959 & 17.959 & 17.959       \\
                       & 30\%                  & 7.10E-02 & 7.10E-02 & 7.10E-02 & 7.10E-02 & 7.10E-02     & 22.974 & 22.974 & 22.974 & 22.975 & 22.975       \\
                       & 40\%                  & 7.16E-02 & 7.16E-02 & 7.16E-02 & 7.16E-02 & 7.16E-02     & 22.919 & 22.920 & 22.920 & 22.920 & 22.920       \\ \hline
\multirow{3}{*}{5000}  & 20\%                  & 1.42E-01 & 1.42E-01 & 1.42E-01 & 1.42E-01 & 1.42E-01     & 16.959 & 16.960 & 16.960 & 16.960 & 16.960       \\
                       & 30\%                  & 1.05E-01 & 1.05E-01 & 1.05E-01 & 1.05E-01 & 1.05E-01     & 19.616 & 19.617 & 19.617 & 19.617 & 19.617       \\
                       & 40\%                  & 8.76E-02 & 8.76E-02 & 8.76E-02 & 8.76E-02 & 8.76E-02     & 21.156 & 21.157 & 21.157 & 21.157 & 21.157       \\ \hline
\multirow{3}{*}{10000} & 20\%                  & 1.46E-01 & 1.46E-01 & 1.46E-01 & 1.46E-01 & 1.46E-01     & 16.715 & 16.716 & 16.716 & 16.716 & 16.716       \\
                       & 30\%                  & 1.11E-01 & 1.11E-01 & 1.11E-01 & 1.11E-01 & 1.11E-01     & 19.106 & 19.107 & 19.107 & 19.107 & 19.107       \\
                       & 40\%                  & 9.26E-02 & 9.26E-02 & 9.26E-02 & 9.26E-02 & 9.26E-02     & 20.673 & 20.673 & 20.673 & 20.674 & 20.674       \\ \hline
\end{tabular}
\end{table}

Next, we test the BDR on another public dataset which has nonsmooth behavior. The dataset is taken from the ``2023 Distribution zone substation data'', published by Ausgrid\footnote{The dataset can be found at \url{https://www.ausgrid.com.au/Industry/Our-Research/Data-to-share/Distribution-zone-substation-data}.}. We choose the data file named ``Beacon Hill 33 11kV FY2023'' for this experiment. This real-world dataset contains the load data of a practical power system in Victoria, Australia for one year, measured at a 15-minute time step. One interesting observation is that the load data becomes sparse after we perform DCT on it, with most of its entries being close to zero. This can be seen in Figure~\ref{fig:load_plot_sparsity}, where 100 samples of the original data versus the sparse values are shown. We repeat the same experiment of the electric voltage signal and report the obtained results in Tables~\ref{tab:results_recon_load_time_iter}\&\ref{tab:results_recon_load_err}. From these two tables, it is observed that the solution quality of the BDR (in terms of the error and SNR criteria) is the same as the solution quality of the pDCAe and HBADMM algorithms and better than the solution quality of the ADMM and the DR-DC algorithms. However, the main advantage of the BDR is in its computational efficiency. It is seen that the BDR converges much faster with a significantly lower CPU time than all other methods and in all test conditions (i.e., all signal length and sampling ratio values) in Tables~\ref{tab:results_recon_load_time_iter}\&\ref{tab:results_recon_load_err}. This computational efficiency is a key issue for power system operators. In severe contingencies (such as when a wildfire happens), some of load data may be lost (for instance, due to the disconnection of some communication lines) and thus just a percentage of load data is available similar to the test cases considered in this numerical experiment. The shorter response time of the BDR means that power systems operators can access faster the reconstructed load profile data and thus can more effectively decide to cope with the severe conditions and supply the customers.  
\begin{figure}[H]
\centering
\includegraphics[width=0.5\linewidth]{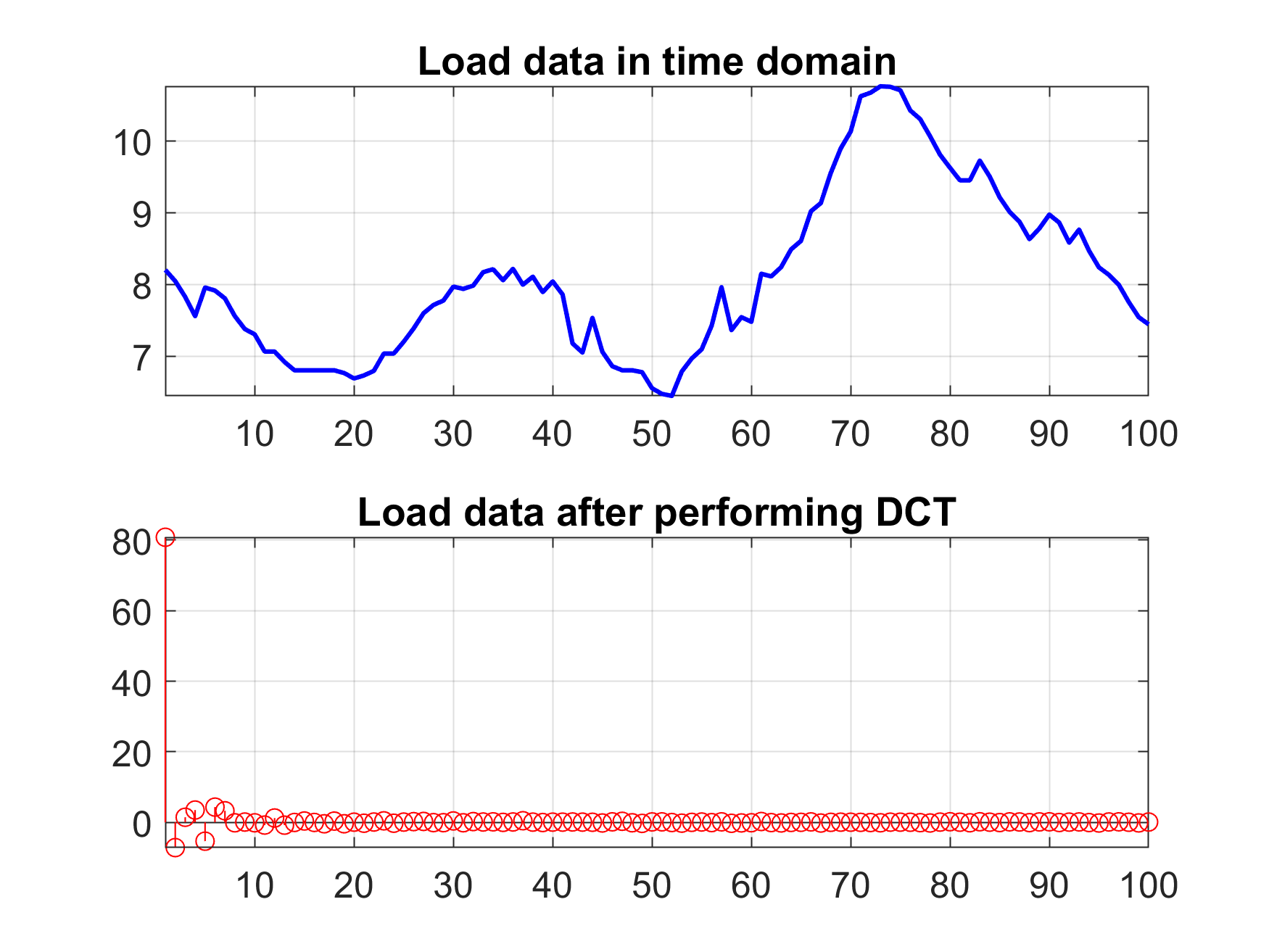}
\caption{Load data and its sparsity.}
\label{fig:load_plot_sparsity}
\end{figure}

\begin{table}[H]
\centering
\caption{Average CPU time and number of iterations over 30 runs, load reconstruction}
\label{tab:results_recon_load_time_iter}
\setlength{\tabcolsep}{1pt}
\begin{tabular}{c|c|ccccc|ccccc}
\hline
\multicolumn{1}{l|}{}  & \multicolumn{1}{l|}{} & \multicolumn{5}{c|}{CPU time (in seconds)}                     & \multicolumn{5}{c}{Number of iterations}               \\ \hline
Length          & $\mathcal{R}$         & ADMM  & DR-DC & pDCAe & HBADMM & \textbf{BDR} & ADMM & DR-DC & pDCAe & HBADMM & \textbf{BDR} \\ \hline
\textbf{}\multirow{3}{*}{2000}  & 20\%      & 1.02   & 1.00   & 0.61  & 0.69   & \textbf{0.12}         & 3000 & 3000  & 385   & \textbf{295}    & 317          \\
                       & 30\%      & 1.95   & 2.04   & 0.48  & 0.91   & \textbf{0.11}         & 3000 & 3000  & 290   & 271    & \textbf{165}          \\
                       & 40\%      & 4.31   & 3.57   & 0.53  & 1.58   & \textbf{0.13}         & 3000 & 2505  & 253   & 295    & \textbf{92}           \\ \hline
\multirow{3}{*}{5000}  & 20\%      & 13.21  & 12.78  & 4.56  & 6.96   & \textbf{1.52}         & 3000 & 3000  & 401   & 391    & \textbf{379}          \\
                       & 30\%      & 20.46  & 20.07  & 3.82  & 9.81   & \textbf{1.36}         & 3000 & 3000  & 331   & 401    & \textbf{198}          \\
                       & 40\%      & 30.53  & 29.31  & 3.49  & 12.72  & \textbf{1.20}         & 3000 & 2839  & 273   & 401    & \textbf{115}          \\ \hline
\multirow{3}{*}{10000} & 20\%      & 49.20  & 46.89  & 15.26 & 25.42  & \textbf{5.16}         & 3000 & 3000  & 399   & 401    & \textbf{329}          \\
                       & 30\%      & 81.74  & 83.73  & 16.79 & 39.81  & \textbf{5.24}         & 3000 & 3000  & 316   & 401    & \textbf{171}          \\
                       & 40\%      & 135.40 & 114.60 & 14.80 & 52.43  & \textbf{4.53}         & 3000 & 2690  & 260   & 397    & \textbf{105}          \\ \hline
\end{tabular}
\end{table}

\begin{table}[H]
\centering
\caption{Average error vs ground truth and SNR over 30 runs, load reconstruction}
\label{tab:results_recon_load_err}
\setlength{\tabcolsep}{1.6pt}
\begin{tabular}{c|c|ccccc|ccccc}
\hline
\multicolumn{1}{l|}{}  & \multicolumn{1}{l|}{} & \multicolumn{5}{c|}{Error vs ground truth}               & \multicolumn{5}{c}{SNR}                          \\ \hline
Length          & $\mathcal{R}$         & ADMM     & DR-DC    & pDCAe    & HBADMM   & \textbf{BDR} & ADMM   & DR-DC  & pDCAe  & HBADMM & \textbf{BDR} \\ \hline
\multirow{3}{*}{2000}  & 20\%      & 2.67E-01 & 6.36E-02 & 5.34E-02 & 5.34E-02 & 5.34E-02     & 11.466 & 23.931 & 25.461 & 25.461 & 25.461       \\
                       & 30\%      & 8.36E-02 & 3.97E-02 & 3.95E-02 & 3.95E-02 & 3.95E-02     & 21.607 & 28.026 & 28.071 & 28.071 & 28.071       \\
                       & 40\%      & 3.08E-02 & 3.03E-02 & 3.03E-02 & 3.03E-02 & 3.03E-02     & 30.227 & 30.384 & 30.386 & 30.387 & 30.387       \\ \hline
\multirow{3}{*}{5000}  & 20\%      & 3.13E-01 & 8.07E-02 & 6.20E-02 & 6.20E-02 & 6.20E-02     & 10.080 & 21.870 & 24.158 & 24.158 & 24.158       \\
                       & 30\%      & 1.22E-01 & 4.37E-02 & 4.33E-02 & 4.33E-02 & 4.33E-02     & 18.273 & 27.195 & 27.269 & 27.269 & 27.269       \\
                       & 40\%      & 3.64E-02 & 3.36E-02 & 3.36E-02 & 3.36E-02 & 3.36E-02     & 28.773 & 29.474 & 29.475 & 29.475 & 29.475       \\ \hline
\multirow{3}{*}{10000} & 20\%      & 3.00E-01 & 7.06E-02 & 5.59E-02 & 5.59E-02 & 5.59E-02     & 10.470 & 23.047 & 25.058 & 25.058 & 25.058       \\
                       & 30\%      & 1.01E-01 & 4.08E-02 & 4.06E-02 & 4.06E-02 & 4.06E-02     & 19.884 & 27.799 & 27.831 & 27.831 & 27.831       \\
                       & 40\%      & 3.31E-02 & 3.16E-02 & 3.16E-02 & 3.16E-02 & 3.16E-02     & 29.604 & 30.006 & 30.007 & 30.007 & 30.007       \\ \hline
\end{tabular}
\end{table}

{\color{black}
\subsection{An analytical example of \eqref{eq:P} with nonconvex $h$}

In this subsection, we evaluate the BDR on a test problem taking the form of \eqref{eq:P}, in which $h$ is nonconvex. Consider the following modified version of \eqref{ex:CSM},
\begin{align}\label{ex:CSM_nonconvex}
    \min_{x\in \mathbb{R}^{d}}~~\frac{1}{2}\|Ax-b\|^2 + \lambda (\|x\|_{1,\mathcal{T}}-\|x\|), \tag{CS-2M}
\end{align}
where $\|x\|_{1,\mathcal{T}}:=\sum_{i=1}^{d}\min(|x_i|,\mathcal{T})$ is the \textit{capped-$\ell_1$} norm \cite{capped_L1}, which is nonconvex. Obviously, when $\mathcal{T}\to +\infty$, the capped-$\ell_1$ becomes the $\ell_1$ norm. Another observation is that the level curves of the $\|x\|_{1,\mathcal{T}}-\|x\|$ have approximately similar patterns to those of the $\|x\|_{1}-\|x\|$, as shown in Figure~\ref{fig:capped_L1}. Therefore, this regularization term can also promote sparsity when $\mathcal{T}$ is set relatively high.
\begin{figure}[H]
\centering
\includegraphics[width=0.8\linewidth]{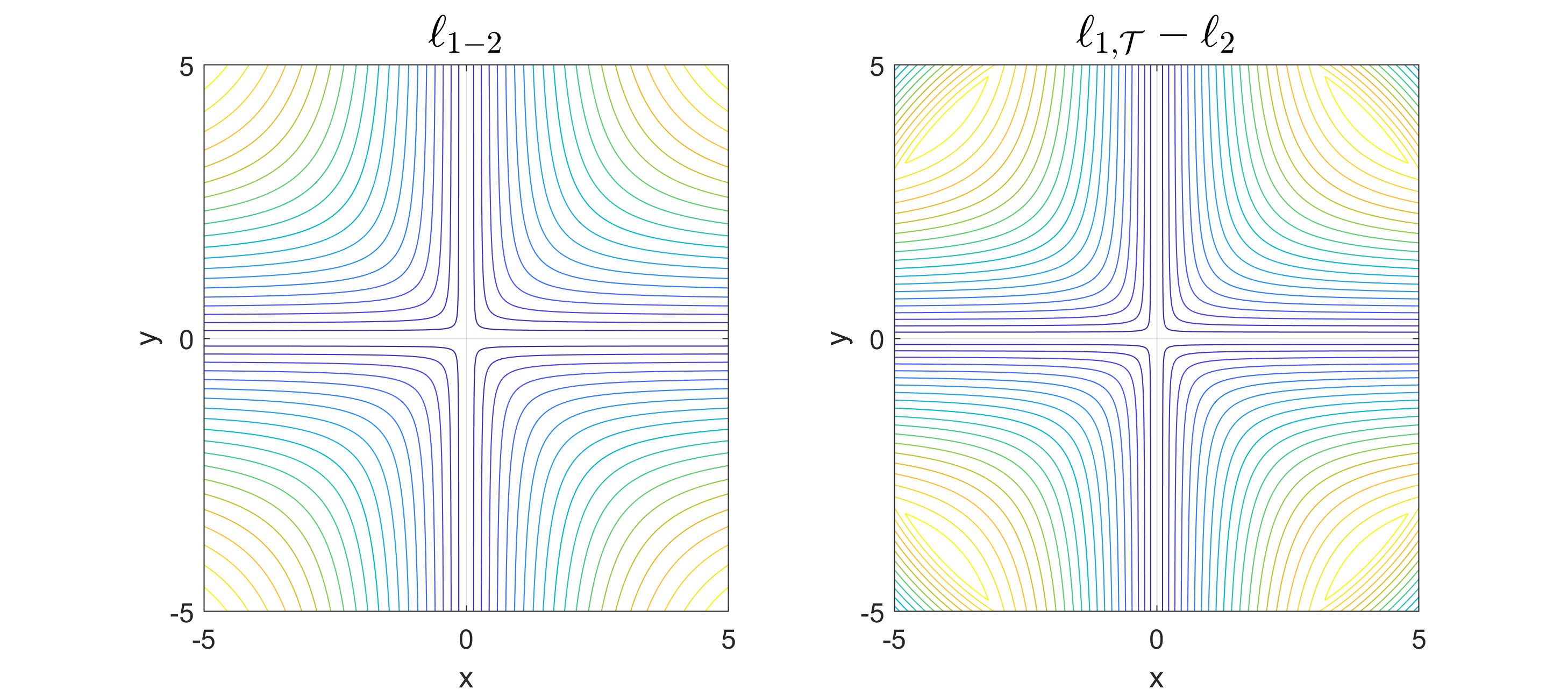}
\caption{Contour lines of $\ell_{1-2}$ vs $\ell_{1,\mathcal{T}}-\ell_2$.}
\label{fig:capped_L1}
\end{figure}
The BDR can be used to solve \eqref{ex:CSM_nonconvex} with the same splitting scheme as introduced in Section~\ref{sec:numerical_result}, now with $h(x)=\lambda\|x\|_{1,\mathcal{T}}$. The closed-form proximal operator for $h$ is given by\footnote{\url{https://proximity-operator.net}}, for $i=1, \dots,d$,
\begin{align*}
(\prox_{\gamma h}(u))_i = \begin{cases}
u_i1_{\left\{|u_i|>\sqrt{2\gamma\lambda\mathcal{T}}\right\}} + \{0,u_i\}1_{\left\{|u_i|=\sqrt{2\gamma\lambda\mathcal{T}}\right\}} \quad &\text{if~} \gamma\lambda/\mathcal{T}\geq 2, \\
\text{sign}(u_i)(|u_i|-\gamma\lambda)1_{\left\{\gamma\lambda < |u_i| < \mathcal{T}+\frac{\gamma\lambda}{2\mathcal{T}}\right\}} + u_i1_{\left\{ |u_i| \geq \mathcal{T} + \frac{\gamma\lambda}{2\mathcal{T}} \right\}} &\text{if~} \gamma\lambda/\mathcal{T} < 2,
\end{cases}
\end{align*}
where $1_{\{x\in C\}} = 1$ if $x\in C$, and $1_{\{x\in C\}} = 0$ otherwise. Again, we generate $A$, $b$, and the ground truth sparse vector by following the procedure in \cite{Lou2017}, and consider different matrix sizes from Table~\ref{tab:caseCS}.  We set $\lambda=0.1$, $\mathcal{T}=100$, and run all algorithms, initialized at the origin, for 30 times with a maximum of 3000 iterations. The stopping condition for all algorithms the same as in previous sections. For the BDR, we also choose the parameters as in Section~\ref{subsec:profile_reconstruction}, with $k=2$. We compare the BDR with the \textit{proximal subgradient algorithm with extrapolation} (PSAE) \cite{Tan2023}, which is capable of tackling \eqref{eq:P} with nonconvex $h$, and its parameters are set to the same ones used in \cite{Tan2023}. Table~\ref{tab: result_nonconvex} shows that the BDR consistently outperforms the PSAE in this nonconvex case study.

\begin{table}[H]
\centering
\caption{Average results over 30 random instances}
\label{tab: result_nonconvex}
\begin{tabular}{c|cc|cc|cc}
\hline
     & \multicolumn{2}{c|}{CPU time (in seconds)} & \multicolumn{2}{c|}{Number of iterations} & \multicolumn{2}{c}{Error vs ground truth} \\ \hline
Case & PSAE            & \textbf{BDR}            & PSAE            & \textbf{BDR}           & PSAE           & \textbf{BDR}             \\ \hline
1    & 0.03             & \textbf{0.01}           & 96               & \textbf{67}            & 7.72E-06        & \textbf{4.96E-06}        \\
2    & 0.17             & \textbf{0.09}           & 92               & \textbf{64}            & 7.45E-06        & \textbf{4.58E-06}        \\
3    & 0.41             & \textbf{0.24}           & 92               & \textbf{64}            & 7.43E-06        & \textbf{4.68E-06}        \\
4    & 0.80             & \textbf{0.47}           & 93               & \textbf{64}            & 7.32E-06        & \textbf{4.63E-06}        \\
5    & 1.20             & \textbf{0.74}           & 92               & \textbf{63}            & 7.31E-06        & \textbf{4.64E-06}        \\
6    & 1.64             & \textbf{1.00}           & 92               & \textbf{64}            & 7.40E-06        & \textbf{4.70E-06}        \\
7    & 2.30             & \textbf{1.47}           & 92               & \textbf{64}            & 7.27E-06        & \textbf{4.52E-06}        \\
8    & 2.44             & \textbf{1.53}           & 92               & \textbf{64}            & 7.39E-06        & \textbf{4.56E-06}        \\ 
9    & 2.95             & \textbf{1.97}           & 91               & \textbf{63}            & 7.27E-06        & \textbf{4.58E-06}        \\
10   & 3.68             & \textbf{2.43}           & 92               & \textbf{63}            & 7.21E-06        & \textbf{4.60E-06}        \\ \hline
11   & 0.03             & \textbf{0.01}           & 93               & \textbf{65}            & 7.67E-06        & \textbf{4.81E-06}        \\
12   & 0.19             & \textbf{0.11}           & 91               & \textbf{63}            & 7.33E-06        & \textbf{4.64E-06}        \\
13   & 0.49             & \textbf{0.40}           & 110              & \textbf{103}           & 9.21E-06        & \textbf{9.00E-06}        \\
14   & 0.68             & \textbf{0.42}           & 92               & \textbf{64}            & 7.41E-06        & \textbf{4.63E-06}        \\
15   & 1.12             & \textbf{0.83}           & 102              & \textbf{81}            & 8.49E-06        & \textbf{6.57E-06}        \\
16   & 1.67             & \textbf{1.07}           & 96               & \textbf{68}            & 7.82E-06        & \textbf{5.02E-06}        \\
17   & 2.19             & \textbf{1.36}           & 91               & \textbf{63}            & 7.17E-06        & \textbf{4.44E-06}        \\
18   & 3.09             & \textbf{2.02}           & 99               & \textbf{74}            & 8.02E-06        & \textbf{5.59E-06}        \\
19   & 3.06             & \textbf{2.05}           & 91               & \textbf{62}            & 7.14E-06        & \textbf{4.47E-06}        \\
20   & 5.93             & \textbf{4.85}           & 117              & \textbf{105}           & 9.83E-06        & \textbf{9.17E-06}        \\ \hline
\end{tabular}
\end{table}

}
\section{Conclusion}
\label{sec:conclusion}

We have proposed a proximal splitting algorithm to solve a class of nonconvex and nonsmooth problems, including the well-known DC program. This algorithm uses proximal steps to evaluate the concave part of the objective function rather than relying on subgradients. We have proved that the sequence of iterates generated by the proposed algorithm is bounded and any of its cluster points yields a {\color{black}critical point} of the problem model. We have also established the global convergence of the whole sequence and derived the convergence rates of the iterates and the objective function value, without requiring differentiability of the concave part. Intensive numerical experiments on both synthetic data and real data shows the superiority of our algorithm. The experiments on electric voltage signals and load data provide a concept for developing a fast data imputation framework for time series data in power systems. The higher computation efficiency of the proposed algorithm can help power system operators to access more rapidly the reconstructed data and therefore better manage their systems. In the future work, we will consider multivariate time series and the spatial relationships between the variables in the time series.

\section*{Declarations}

\noindent\textbf{Acknowledgements}. The authors want to thank Prof. Hongjin He for sharing the codes in \cite{Chuang2021}. We also want to thank Dr. Kai Tu and Dr. Emmanuel Soubies for the insightful discussions during the development of this manuscript. {\color{black}We are grateful to the anonymous reviewers for their constructive comments and suggestions which helped improve the manuscript.}

\noindent\textbf{Funding}. The research of TNP was supported by Henry Sutton PhD Scholarship Program from Federation University Australia. The research of MND was supported by Grant DP230101749 from the Australian Research Council (ARC) and benefited from the FMJH Program Gaspard Monge for optimization and operations research and their interactions with data science.

\noindent\textbf{Conflict of interest}. The authors declare no competing interests.

\noindent\textbf{Availability of data and materials}. All data generated or analyzed during this study are included in this article. The data for the numerical experiments were generated randomly and we explained explicitly how to generate them.

\noindent\textbf{Authors' contributions}. All authors contributed to the manuscript and approved the submitted version.


\let\oldbibliography\thebibliography
\renewcommand{\thebibliography}[1]{%
  \oldbibliography{#1}%
  \setlength{\itemsep}{-2pt}%
}
\bibliographystyle{plain}
\bibliography{references.bib}

\begin{thebibliography}{10}

\bibitem{An2016}
N.~T. An and N.~M. Nam.
\newblock Convergence analysis of a proximal point algorithm for minimizing differences of functions.
\newblock {\em Optimization}, 66(1):129--147, 2016.

\bibitem{AragnArtacho2020}
F.~J. Aragón~Artacho and P.~T. Vuong.
\newblock The boosted difference of convex functions algorithm for nonsmooth functions.
\newblock {\em SIAM Journal on Optimization}, 30(1):980–1006, 2020.

\bibitem{Attouch2007}
H.~Attouch and J.~Bolte.
\newblock On the convergence of the proximal algorithm for nonsmooth functions involving analytic features.
\newblock {\em Mathematical Programming}, 116(1–2):5–16, 2007.

\bibitem{Attouch2011}
H.~Attouch, J.~Bolte, and B.~F. Svaiter.
\newblock Convergence of descent methods for semi-algebraic and tame problems: proximal algorithms, forward–backward splitting, and regularized gauss–seidel methods.
\newblock {\em Mathematical Programming}, 137(1–2):91–129, 2011.

\bibitem{Babakmehr2018}
M.~Babakmehr, M.~Majidi, and M.~G. Simoes.
\newblock Compressive sensing for power system data analysis.
\newblock In {\em Big Data Application in Power Systems}, page 159–178. Elsevier, 2018.

\bibitem{Banert2018}
S.~Banert and R.~I. Boț.
\newblock A general double-proximal gradient algorithm for d.c. programming.
\newblock {\em Mathematical Programming}, 178(1–2):301–326, 2018.

\bibitem{Bauschke2017_book}
H.~Bauschke and P.~Combettes.
\newblock {\em Convex analysis and monotone operator theory in Hilbert spaces}.
\newblock CMS Books in Mathematics. Springer International Publishing, Basel, Switzerland, 2 edition, 2017.

\bibitem{BDL22}
R.~I. Bo{\c{t}}, M.~N. Dao, and G.~Li.
\newblock Extrapolated proximal subgradient algorithms for nonconvex and nonsmooth fractional programs.
\newblock {\em Mathematics of Operations Research}, 47(3):2415--2443, 2022.

\bibitem{Chuang2021}
C-S. Chuang, H.~He, and Z.~Zhang.
\newblock A unified {D}ouglas–{R}achford algorithm for generalized {DC} programming.
\newblock {\em Journal of Global Optimization}, 82(2):331–349, 2021.

\bibitem{Davis2017}
D.~Davis and W.~Yin.
\newblock A three-operator splitting scheme and its optimization applications.
\newblock {\em Set-Valued and Variational Analysis}, 25(4):829–858, 2017.

\bibitem{deOliveira2025}
W.~de~Oliveira and J.~C. Souza.
\newblock A progressive decoupling algorithm for minimizing the difference of convex and weakly convex functions.
\newblock {\em Journal of Optimization Theory and Applications}, 204(3), 2025.

\bibitem{DR56}
J.~Douglas and H.~H. Rachford.
\newblock On the numerical solution of heat conduction problems in two and three space variables.
\newblock {\em Transactions of the American Mathematical Society}, 82(2):421–439, 1956.

\bibitem{Kazda2024}
K.~Kazda and X.~Li.
\newblock A linear programming approach to difference-of-convex piecewise linear approximation.
\newblock {\em European Journal of Operational Research}, 312(2):493–511, 2024.

\bibitem{Kurdyka1998}
K.~Kurdyka.
\newblock On gradients of functions definable in o-minimal structures.
\newblock {\em Annales de l'institut Fourier}, 48(3):769--783, 1998.

\bibitem{Le_Thi2018-ht}
H.~A. Le~Thi and T.~Pham~Dinh.
\newblock {DC} programming and {DCA}: thirty years of developments.
\newblock {\em Mathematical Programming}, 169(1):5--68, 2018.

\bibitem{Li2015DR}
G.~Li and T.~K. Pong.
\newblock Douglas–{R}achford splitting for nonconvex optimization with application to nonconvex feasibility problems.
\newblock {\em Mathematical Programming}, 159(1–2):371–401, 2015.

\bibitem{Liu2023OR}
J.~Liu and J-S. Pang.
\newblock Risk-based robust statistical learning by stochastic difference-of-convex value-function optimization.
\newblock {\em Operations Research}, 71(2):397–414, 2023.

\bibitem{Loja63}
S.~{\L}ojasiewicz.
\newblock Une propriété topologique des sous-ensembles analytiques réels.
\newblock {\em Les Équations aux Dérivées Partielles}, pages 87--89, 1963.

\bibitem{Lou2017}
Y.~Lou and M.~Yan.
\newblock Fast {L1}–{L2} minimization via a proximal operator.
\newblock {\em Journal of Scientific Computing}, 74(2):767–785, 2017.

\bibitem{Nejati2012}
M.~Nejati, N.~Amjady, and H.~Zareipour.
\newblock A new stochastic search technique combined with scenario approach for dynamic state estimation of power systems.
\newblock {\em IEEE Transactions on Power Systems}, 27(4):2093–2105, 2012.

\bibitem{Nesterov2018}
Y.~Nesterov.
\newblock {\em Lectures on Convex Optimization}, volume 137 of {\em Springer Optimization and Its Applications}.
\newblock Springer International Publishing, Cham, Switzerland, 2018.

\bibitem{Tan2023}
T.~N. Pham, M.~N. Dao, R.~Shah, N.~Sultanova, G.~Li, and S.~Islam.
\newblock A proximal subgradient algorithm with extrapolation for structured nonconvex nonsmooth problems.
\newblock {\em Numerical Algorithms}, 94(4):1763–1795, 2023.

\bibitem{Tao1986}
T.~Pham~Dinh and E.~B. Souad.
\newblock {\em Algorithms for Solving a Class of Nonconvex Optimization Problems. Methods of Subgradients}, page 249–271.
\newblock Elsevier, 1986.

\bibitem{RW98}
R.~T. Rockafellar and R.~{J-B. Wets}.
\newblock {\em {Variational Analysis}}, volume 317 of {\em Grundlehren der mathematischen Wissenschaften}.
\newblock Springer Berlin Heidelberg, Berlin, Heidelberg, 1998.

\bibitem{CSAPPS}
R.~Sankararajan, H.~Rajendran, and A.~N. Sukumaran.
\newblock 11 real-world applications of compressive sensing.
\newblock In {\em Compressive Sensing for Wireless Communication: Challenges and Opportunities}, pages 413--446. 2016.

\bibitem{Sun2017}
T.~Sun, P.~Yin, L.~Cheng, and H.~Jiang.
\newblock Alternating direction method of multipliers with difference of convex functions.
\newblock {\em Advances in Computational Mathematics}, 44(3):723--744, 2017.

\bibitem{PDCA_first}
W.~Sun, R.~J.~B. Sampaio, and M.~A.~B. Candido.
\newblock Proximal point algorithm for minimization of {DC} function.
\newblock {\em Journal of Computational Mathematics}, 21(4):451--462, 2003.

\bibitem{bundle_DC}
K.~Syrtseva, W.~de~Oliveira, S.~Demassey, and W.~van Ackooij.
\newblock Minimizing the difference of convex and weakly convex functions via bundle method.
\newblock {\em Pacific Journal of Optimization}, 20(4):699--741, 2024.

\bibitem{Tibshirani1996}
R.~Tibshirani.
\newblock Regression shrinkage and selection via the {L}asso.
\newblock {\em Journal of the Royal Statistical Society: Series B (Methodological)}, 58(1):267–288, 1996.

\bibitem{Tu2019}
K.~Tu, H.~Zhang, H.~Gao, and J.~Feng.
\newblock A hybrid {Bregman} alternating direction method of multipliers for the linearly constrained difference-of-convex problems.
\newblock {\em Journal of Global Optimization}, 76(4):665--693, 2019.

\bibitem{vanAckooij2019}
W.~van Ackooij and W.~de~Oliveira.
\newblock Some brief observations in minimizing the sum of locally lipschitzian functions.
\newblock {\em Optimization Letters}, 14(3):509–520, 2019.

\bibitem{Wen2017}
B.~Wen, X.~Chen, and T.~K. Pong.
\newblock A proximal difference-of-convex algorithm with extrapolation.
\newblock {\em Computational Optimization and Applications}, 69(2):297--324, 2017.

\bibitem{capped_L1}
T.~Zhang.
\newblock Multi-stage convex relaxation for learning with sparse regularization.
\newblock In {\em Advances in Neural Information Processing Systems}, volume~21, 2008.

\end{thebibliography}

\end{document}